\documentclass{article}

\usepackage{hyperref}
\usepackage{amsmath}
\usepackage{amssymb}
\usepackage{amsthm}
\usepackage{tikz}
\usetikzlibrary{arrows.meta}

\newtheorem{lemma}{Lemma}
\newtheorem{theorem}{Theorem}
\newtheorem*{theorem*}{Theorem}

\newcommand{\dotoplus}{ \mathbin{ \dot{\oplus} } }
\newcommand{\inv}[1]{
	\!\;
	\overline{
		\!\!\: #1 \vphantom{!} \!\!\:
	}
	\;\!
}
\newcommand{\leqt}{ \trianglelefteq }
\newcommand{\up}[2]{ { ^{#1} \! {#2} } }

\DeclareMathOperator{\Image}{Im}
\DeclareMathOperator{\Aut}{Aut}
\DeclareMathOperator{\Hom}{Hom}

\DeclareMathOperator{\Pro}{Pro}
\DeclareMathOperator{\Ind}{Ind}
\newcommand{\Ring}{ \mathbf{Ring} }
\newcommand{\Set}{ \mathbf{Set} }
\newcommand{\Cat}{ \mathbf{Cat} }

\DeclareMathOperator{\Mat}{M}
\DeclareMathOperator{\Ker}{Ker}
\DeclareMathOperator{\Center}{Z}

\newcommand{\sMat}[4]{
	\bigl(
		\begin{smallmatrix}
				#1 \mathstrut
			&
				#2 \mathstrut
			\\
				#3 \mathstrut
			&
				#4 \mathstrut
		\end{smallmatrix}
	\bigr)
}

\DeclareMathOperator{\Spec}{Spec}
\DeclareMathOperator{\Max}{Max}
\DeclareMathOperator{\Jac}{J}
\DeclareMathOperator{\Closure}{Cl}
\DeclareMathOperator{\jdim}{jdim}
\newcommand{\ClosedSet}{ \mathop{ \mathcal{V} } }

\newcommand{\Sym}{ \mathrm{S} }
\newcommand{\Field}{ \mathbb{F} }
\newcommand{\Real}{ \mathbb{R} }
\newcommand{\Int}{ \mathbb{Z} }

\newcommand{\fppf}{ \mathrm{fppf} }
\newcommand{\Affine}{ \mathbb{A} }
\newcommand{\ModSheaf}{ \mathop{ \mathbb{W} } }
\DeclareMathOperator{\Restrict}{R}

\newcommand{\Presheaf}{ \mathbf{P} }
\newcommand{\R}{ \mathcal{R} }

\newcommand{\e}{ \mathrm{e} }
\newcommand{ \RootSys }[2]{ \mathsf{#1}_{#2} }
\newcommand{ \TitsIndex }[5]{
	\up{#2}{ \mathsf{#1}_{#3, #4}^{#5} }
}

\DeclareMathOperator{\Heis}{Heis}
\DeclareMathOperator{\Hyp}{H}
\newcommand{\OddFormPar}{ \mathcal{L} }
\newcommand{\hdisc}{ \mathop{ \mathrm{hdisc} } }

\DeclareMathOperator{\GenLin}{GL}
\DeclareMathOperator{\SpecLin}{SL}
\DeclareMathOperator{\Elem}{E}
\DeclareMathOperator{\KFunc}{K}
\DeclareMathOperator{\CongLin}{C}

\DeclareMathOperator{\Orth}{O}
\DeclareMathOperator{\SpecOrth}{SO}
\DeclareMathOperator{\Refl}{Refl}
\newcommand{ \SpecOrthSch }{ \mathbb{SO} }
\newcommand{ \ProjSpecOrthSch }{ \mathbb{PSO} }

\DeclareMathOperator{\Unit}{U}
\DeclareMathOperator{\SpecUnit}{SU}
\DeclareMathOperator{\ElemUnit}{EU}
\DeclareMathOperator{\KUnit}{KU}
\newcommand{\UnitSch}{ \mathbb{U} }

\newcommand{\liealg}{ \mathfrak{g} }
\newcommand{\simpcon}{ \mathrm{sc} }
\newcommand{\IsoPinning}{ \mathcal{U} }
\DeclareMathOperator{\ChevGrp}{G}
\DeclareMathOperator{\IsoPin}{IsoPin}

\title{
	Nilpotency of locally isotropic
	\( \KFunc_1 \)-functor
}

\author{
	Egor Voronetsky%
	\thanks{
		This research is supported by the Russian Science Foundation grant 26-11-00085.
	} \\
	Saint Petersburg University, \\
	7/9 Universitetskaya nab., \\
	St. Petersburg, 199034 Russia
}

\begin{document}
\maketitle

\begin{abstract}
	We show that the
	\( \KFunc_1 \)-functor modeled on locally isotropic reductive groups is hypoabelian if the base ring has finite Bass--Serre dimension and, for the Tits index
	\( \TitsIndex{E}{}{8}{2}{78} \), contains a field. For classical or globally isotropic reductive groups schemes
	\( \KFunc_1 \) is actually solvable. This implies that the elementary subgroup (or its derived subgroup) is the maximal perfect subgroup of the reductive group.
\end{abstract}

\section{Introduction}

Let
\( R \) be an associative ring and
\( \GenLin(n, R) \) the general linear group over
\( R \) for
\( n \geq 3 \). It is well known that if
\( R = D \) is a division ring, then the derived subgroup
\( \bigl[ \GenLin(n, D), \GenLin(n, D) \bigr] \) is simple modulo its center. This implication also holds for
\( n = 2 \) with only two exceptions
\( D \cong \Field_2 \) and
\( D \cong \Field_3 \), because in the exceptional cases the group
\( \GenLin(n, D) \) is already solvable.

For arbitrary rings the situation is different. First of all, let us introduce the
\textit{elementary subgroup}
\[
	\Elem(n, R)
	=
	\bigl\langle
		1 + x e_{i j}
		\mid
		i \neq j,\ %
		x \in R
	\bigr\rangle
	\leq
	\GenLin(n, R).
\]
This subgroup is always perfect, i.e.\ %
\(
	\Elem(n, R)
	=
	\bigl[ \Elem(n, R), \Elem(n, R) \bigr]
\). The center of
\( \GenLin(n, R) \) coincides with the centralizer of
\( \Elem(n, R) \) and consists of scalar matrices, so it is isomorphic to the group of invertible elements of the center of
\( R \).

The elementary subgroup is far from simple, but its normal structure is easy to describe up to certain homological invariants of
\( R \). Namely, if
\( R \) is
\textit{almost commutative} (i.e.\ every finite subset of
\( R \) generates a module finite algebra over the center of
\( R \)), then a subgroup
\( H \leq \GenLin(n, R) \) is normalized by
\( \Elem(n, R) \) if and only if there exists a (necessarily unique) ideal
\( I \leqt R \) such that
\[
	\Elem(n, R, I) \leq H \leq \CongLin(n, R, I).
\]
Here the
\textit{relative elementary subgroup}
\( \Elem(n, R, I) \) is the normal subgroup of
\( \Elem(n, R) \) generated by
\( 1 + x e_{i j} \) for
\( i \neq j \) and
\( x \in I \). The
\textit{full congruence subgroup}
\( \CongLin(n, R, I) \) is the preimage of the center of
\( \GenLin(n, R / I) \) under the canonical homomorphism
\( \GenLin(n, R) \to \GenLin(n, R / I) \). Moreover,
\( \Elem(n, R, I) \) (including
\( \Elem(n, R) \)) is normal in
\( \GenLin(n, R) \) and the following
\textit{standard commutator formulae} hold.
\[
	\bigl[ \GenLin(n, R), \Elem(n, R, I) \bigr]
	=
	\bigl[ \CongLin(n, R, I), \Elem(n, R) \bigr]
	=
	\Elem(n, R, I).
\]
Thus the description of subgroups normalized by the elementary subgroup reduces to the classification of subgroups of the factor-groups
\( \CongLin(n, R, I) / \Elem(n, R, I) \). Let
\( \GenLin(n, R, I) \) be the kernel of the canonical homomorphism
\( \GenLin(n, R) \to \GenLin(n, R / I) \), i.e.\ the group of invertible matrices congruent to the identity matrix modulo
\( I \). The factor-group
\( \CongLin(n, R, I) / \GenLin(n, R, I) \) clearly embeds into the center of
\( \GenLin(n, R / I) \), so it is abelian.

Finally, Anthony Bak proved in
\cite{k1-nil-lin} that the
\textit{%
	relative
	\( \KFunc_1 \)-functor%
}
\[
	\KFunc_1(n, R, I) = \GenLin(n, R, I) / \Elem(n, R, I)
\] is solvable provided that the
\textit{Bass--Serre dimension} of the center of
\( R \) is finite. More precisely, without the dimension assumption there is a natural
\textit{dimension filtration}
\[
	\GenLin(n, R, I)
	=
	\GenLin^{- \infty}(n, R, I)
	\geq
	\GenLin^0(n, R, I)
	\geq
	\GenLin^1(n, R, I)
	\geq
	\GenLin^2(n, R, I)
	\geq
	\ldots
	\geq
	\Elem(n, R, I)
\]
such that
\begin{align*}
		\GenLin^i(n, R, I) &\leqt \GenLin(n, R),
	\\
		\bigl[
			\GenLin(n, R, I),
			\GenLin(n, R, I)
		\bigr]
		&\leq
		\GenLin^0(n, R, I),
	\\
		\bigl[
			\GenLin^0(n, R, I),
			\GenLin^i(n, R, I)
		\bigr]
		&\leq
		\GenLin^{i + 1}(n, R, I)
		\text{ for }
		i \geq 0.
\end{align*}
If the Bass--Serre dimension
\( \delta \) of the center of
\( R \) is finite (or
\( - \infty \) in the case
\( R = 0 \)), then
\( \GenLin^\delta(n, R, I) = \Elem(n, R, I) \) and the filtration stabilizes. In this case
\( \KFunc_1(n, R, I) \) is an extension of the nilpotent group
\( \GenLin^0(n, R, I) / \Elem(n, R, I) \) by the abelian group
\( \GenLin(n, R, I) / \GenLin^0(n, R, I) \), so this group is solvable. In particular,
\( \Elem(n, R) \) is the largest perfect subgroup of
\( \GenLin(n, R) \). If
\( R \) is commutative, then
\( \GenLin^0(n, R) = \SpecLin(n, R) \). There are explicit examples
\cite[corollary 7.4]{k1-nil-lin} of commutative rings with finite Bass--Serre dimension and arbitrary large nilpotent class of
\( \mathrm{SK}_1(n, R) = \SpecLin(n, R) / \Elem(n, R) \).

The Bass--Serre dimension is a refinement of the ordinary Krull dimension, in particular, it is always zero for semilocal rings. For the first time this notion implicitly appears in the work
\cite{stab-lin} of Hyman Bass, where the upper bound on the dimension is used to generalize a theorem of Serre on existence of free summands in locally sufficiently isotropic modules over module finite algebras over
\( K \). Also the Bass--Serre dimension is used to bound the number of generators of finitely generated modules
\cite{gen-num} and give upper bounds on various stable ranks
\cite{stab-odd}. Formally these two papers use the \textit{Jacobson dimension}
\( \jdim(K) = \dim(\Max(K)) \) instead of
\( \delta(K) \), but the arguments inside work with
\( \delta(K) \) as well.

The above solvability result is commonly known as
\textit{%
	nilpotency of
	\( \KFunc_1 \)-functor%
}. Later similar results were proved by Bak, Roozbeh Hazrat and Nikolai Vavilov for Chevalley groups \cite{k1-nil-rel, k1-nil-che} and even unitary groups \cite{k1-nil-rel, k1-nil-uni}. Recently the last result was generalized to a wider class of odd unitary groups \cite{k1-nil-odd} under an additional assumption.

We are interested in point groups
\( G(K) \) of reductive group schemes
\( G \) with a local isotropicity condition over a commutative ring
\( K \), i.e.\ ``isotropic reductive groups''. Split reductive groups over rings, also called Chevalley groups, are well studied. Their normal structure is described by Eiichi Abe in
\cite{nor-str-che} except the root systems
\( \RootSys{B}{2} \) and
\( \RootSys{G}{2} \), and by Douglas Costa and Gordon Keller
\cite{nor-str-sp, nor-str-g2} in the remaining cases. Victor Petrov and Anastasia Stavrova started to study isotorpic reductive groups in
\cite{ele-iso}, in particular, they defined the elementary subgroup
\( \Elem_G(K) \) and proved its normality assuming that the isotropic rank of
\( G_{\mathfrak{p}} \) is at least
\( 2 \) at every prime ideal
\( \mathfrak{p} \leqt K \) and there is a proper parabolic subgroup
\( P \leq G \). Later Stavrova and Alexander Luzgarev showed
\cite{ele-per-iso} that
\( \Elem_G(K) \) is perfect unless
\( K \) has residue fields isomorphic to
\( \Field_2\) and simultaneously the absolute root system of
\( G \) has components of types
\( \RootSys{B}{2} \) or
\( \RootSys{G}{2} \). Also Stavrova and Alexey Stepanov
\cite{nor-str-red} proved the normal structure theorem for
\( G(K) \) under the same assumptions plus invertibility of small integer primes in
\( K \) depending on the absolute root system of
\( G \). Recently Leonid Danilevich
\cite{nor-str-odd} generalized the normal structure theorem for ``globally isotropic'' odd orhogonal groups without invertibility of
\( 2 \), so the only remaining ``globally isotropic'' case in this problem is the absolute root system
\( \RootSys{C}{\ell} \).

The functor
\( \KFunc_1 \) of isotropic reductive groups over fields is well studied. Its triviality for simple simple connected groups is known as the Kneser--Tits problem. Though in general there are counterexamples, this problem is affirmatively solved in a number of cases. See
\cite{kne-tit} for general discussion and
\cite{tit-wei-acp, tit-wei-tha} for later results about the particular Tits index
\( \TitsIndex{E}{}{8}{2}{78} \) (the Tits--Weiss conjecture). These results were generalized to semilocal algebras over fields in the preprint
\cite{k1-red-tri}.

In
\cite{ele-loc} we started to study the elementary subgroup
\( \Elem_G \) without assuming existence of a proper parabolic subgroup, so even its definition is non-trivial. In particular, we gave a functorial construction of this subgroup and proved its normality and perfectness under natural assumptions. For example, if
\( P \) is a finitely generated projective module of rank at least
\( 3 \) at every prime ideal, then its group of linear automorphisms
\( \GenLin(P) \) is reductive and of isotropic rank at least
\( 2 \) at every prime ideal, but if
\( P \) is indecomposable, then this group has no proper parabolic subgroups.

In this paper our goal it to prove the following nilpotency theorem for locally isotropic
\( \KFunc_1 \)-functor.
\begin{theorem*}
	Let
	\( K \) be a commutative unital ring and
	\( G \) a reductive group scheme over
	\( K \). Suppose that the Bass--Serre dimension of
	\( K \) is finite, the local isotropic rank of
	\( G \) is at least
	\( 2 \), and
	\( K \) contains a field if the Tits index
	\( \TitsIndex{E}{}{8}{2}{78} \) appears in a maximal isotropic pinning of
	\( G_{ \mathfrak{m} } \). Then
	\( \KFunc_1^G(K) \) is hypoabelian, i.e.\ %
	\( [ \Elem_G(K), \Elem_G(K) ] \leq G(K) \) is the largest perfect subgroup. Moreover,
	\( \KFunc_1^G(K) \) is solvable if in addition
	\begin{itemize}
		
		\item
		either
		\( G \) has an isotropic pinning of rank at least
		\( 2 \) (and
		\( K \) contains a field if
		\( \TitsIndex{E}{}{8}{2}{78} \) appears in this isotropic pinning)
		
		\item
		or
		\( G \) is simple with classical absolute root system (and avoids triality if this root system is of type
		\( \RootSys{D}{4} \)).
		
	\end{itemize}
\end{theorem*}

The condition on the isotropic rank is the exact analogue of the inequality
\( n \geq 3 \) from the linear case.

We consider only the absolute case for several reasons. First of all, not all
\textit{levels} (i.e.\ parameters classifying relative elementary subgroups) are ideals of
\( K \) if
\( 6 \) is not invertible and the absolute root system of
\( G \) is not simply laced. Next, the normal structure theorem itself is not proved in full generality even over local rings (there remains the case of twisted symplectic groups with non-invertible
\( 2 \in K \) and isotropic rank
\( 2 \)), so we have only a conjectural list of possible levels in some cases. Finally, if
\( G = \mathrm{Sp}(4, K) \) and
\( K \) has residue fields isomorphic to
\( \Field_2 \), then the normal structure theorem uses radices \cite{nor-str-sp}, and these algebraic structures are highly complicated compared to ideals and Abe's
\textit{admissible pairs}
\cite{nor-str-che}.

The paper is organized as follows. In \S 2 we recall the definition of Bass--Serre dimension. Sections 3--5 contain general results about elementary subgroups of locally isotropic reductive groups, they form a natural continuation of
\cite{ele-loc}. In particular, in theorem
\ref{perfect} we prove that the derived subgroup
\( [ \Elem_G(K), \Elem_G(K) ] \) is always perfect even if
\( \Elem_G(K) \) itself is not, this fact is well known for Chevalley groups. In \S 6 we prove our first main results about nilpotency of
\( \KFunc_1 \), namely, we construct the dimension filtration and show all its basic properties except solvability of
\( G(K) / G^0(K) \) in theorem
\ref{sol-dim}. We use Bak's localization-completion method in lemma
\ref{loc-com}, but direct calculations from Bak's work are replaced by abstract arguments using colocalization as in
\cite{ele-loc}.

Solvability of
\( \KFunc_1 \) over semilocal rings is considered in remaining sections. In \S 7 we recall some general theory of odd unitary groups, including explicit description of all simple group schemes with classical absolute root systems up to isogeny (lemma
\ref{twi-cla}) and with a classical isotropic pinning up to isogeny (theorem
\ref{twi-iso-cla}). If the absolute root system has type
\( \mathsf{D}_4 \), we have to impose additional condition that the group scheme actually has a corresponding Azumaya algebra with involution and an odd form parameter, i.e.\ that the group scheme
\textit{avoids triality}. In theorem
\ref{sol-sem-iso} from \S 8 we prove that
\( \KFunc_1 \) is solvable over semilocal rings if the group scheme is sufficiently globally isotropic, e.g.\ if the ring is local. This is done case by case. Almost all Tits indices are reduced to various facts about unitary groups, but in the case
\( \TitsIndex{E}{}{8}{2}{78} \) we refer to
\cite{k1-red-tri} thus getting the additional condition in the main theorem. If the group scheme over a semilocal ring is only locally isotropic, then we still prove that
\( \KFunc_1 \) is solvable in theorem
\ref{sol-card}, but the solvability length depends on the number of maximal ideals. We also prove a uniform bound on the solvability length in theorem
\ref{sol-sem-loc}, but only for classical groups. Finally, in the last section 10 we give the main result (theorem
\ref{sol-main}) with a couple of remarks and a list of open problems.

The main technical differences in comparison with previous results about nilpotency of
\( \KFunc_1 \) are the following. We give an alternative conceptual proof of the localization-completion method via colocalization instead of direct calculation. Solvability of globally isotropic
\( \KFunc_1 \) over semilocal rings requires additional work, especially for
\( \TitsIndex{E}{}{7}{2}{31} \), and the full generality of
\( \TitsIndex{E}{}{8}{2}{78} \) is still open. Solvability for arbitrary locally isotropic reductive
\( \KFunc_1 \) over semilocal rings actually uses additional
\textit{cardinality filtration} introduced in \S 9. Finally, a uniform bound on the solvability length (for classical locally isotropic groups) uses a separate non-trivial argument from lemmas
\ref{spl-lin} and
\ref{spl-uni}, this is a further development of the localization technique from
\cite{ele-loc}.

The author wants to thank Anastasia Stavrova for giving motivation to work on this problem.

\section{Bass--Serre dimension}

In this paper we fix a commutative unital ring
\( K \). Its maximum spectrum
\( \Max(K) \) is the set of its maximal ideals with the Zariski topology, i.e.\ principal open subsets
\[
	\mathcal{D}(s)
	=
	\{
		\mathfrak{m} \in \Max(K)
		\mid
		s \notin \mathfrak{m}
	\}
\]
form a basis of this topology and closed sets are precisely the sets of form
\[
	\ClosedSet( \mathfrak{a} )
	=
	\{
		\mathfrak{m} \in \Max(K)
		\mid
		\mathfrak{a} \leq \mathfrak{m}
	\}
\]
for various ideals
\( \mathfrak{a} \leqt K \). Unlike the ordinary spectrum
\( \Spec(K) \), this construction is not functorial in general, but if
\( f \colon K \to E \) is an integral homomorphism, then there is a well-defined continuous map
\[
	\Max(f) \colon \Max(E) \to \Max(K),\,
	\mathfrak{m} \mapsto f^{- 1}( \mathfrak{m} ).
\]

Recall that a topological space
\( X \) is called
\textit{Noetherian} if every increasing chain of its open subsets stabilizes, equivalently, if every decreasing chain of its closed subsets stabilizes. If
\( X \) is Noetherian, then all its subspaces are Noetherian. Conversely, if
\( X = \bigcup_{i = 1}^N \) for Noetherian subspaces
\( X_i \subseteq X \), then
\( X \) is Noetherian. So we have implications
\[
	K \text{ is Noetherian}
	\Rightarrow
	\Spec(K) \text{ is Noetherian}
	\Rightarrow
	\Max(K) \text{ is Noetherian}.
\]

A Noetherian space
\( X \) is
\textit{irreducible} if it is non-empty and it cannot be decomposed into a union of two proper closed subsets. Every Noetherian space
\( X \) has a unique decomposition
\( X = \bigcup_{i = 1}^N X_i \) into its
\textit{irreducible components}, i.e.\ %
\( X_i \) are irreducible closed subsets of
\( X \) and
\( X_i \not \subseteq X_j \) for
\( i \neq j \). Clearly, every irreducible subset of
\( X \) has irreducible closure and it is contained in one of the components
\( X_i \). Since irreducible spaces are connected,
\( X \) has also a unique decomposition into finitely many
\textit{connected components}. So if
\( \Max(K) \) is Noetherian, then there is a unique decomposition
\[
	K = K_1 \times \ldots \times K_n
\]
into a product of indecomposable rings.

The
\textit{combinatorial dimension}
\( \dim(X) \) of a Noetherian space
\( X \) is the supremum of
\( d \geq 0 \) such that there exists a chain
\[
	X_0
	\subset
	X_1
	\subset
	\ldots
	\subset
	X_d
	\subseteq
	X
\]
of distinct closed irreducible subsets in
\( X \). We also define
\( \dim(\varnothing) = - \infty \) and
\( \dim(X) = + \infty \) for non-Noetherian
\( X \). Now let
\( \delta(X) \) be the smallest
\( d \) such that
\( X \) has a cover by finitely many subspaces of dimension at most
\( d \) (they can be assumed to be irreducible, but not necessarily closed). In particular,
\( \delta(X) \geq 0 \) unless
\( X = \varnothing \) (and
\( \delta(\varnothing) = - \infty \)), and
\( \delta(X) = + \infty \) if
\( X \) is not Noetherian.

The
\textit{Krull dimension} of
\( K \) is
\( \dim(K) = \dim(\Spec(K)) \), the
\textit{Jacobson dimension} (using the terminology from
\cite{k1-nil-rel}) is
\( \jdim(K) = \dim(\Max(K)) \), and, finally, the \textit{Bass--Serre dimension} is
\( \delta(K) = \delta(\Max(K)) \). We have
\[
	\delta(K) \leq \jdim(K) \leq \dim(K),
\]
here both inequalities can be strict. Indeed, if
\( K \) is a local ring, then
\( \Max(K) \) is a one-point space and
\( \delta(K) = \jdim(K) = 0 \), but
\( \dim(K) \) can be arbitrary non-negative. On the other hand, if
\( K \) is Noetherian of finite Krull dimension
\( \dim(K) = d \) and
\( \dim(K / \Jac(K)) < d \) (e.g.\ %
\( K = \mathbb{C}[[ s_1, \ldots, s_d ]] \)), then
\( \jdim(K[t]) = \dim(K[t]) = d + 1 \), but
\( \delta(K[t]) \leq d \) by
\cite[proposition III.3.13]{k-theory}. Actually, it is easy to construct examples with any properties of
\( \Max(K) \) using that a topological space
\( X \) is homeomorphic to the maximal spectrum of a commutative unital ring if and only if
\( X \) is
\( \mathrm{T}_1 \) and quasi-compact
\cite[proposition 11]{spectra}.

The next theorem collects some basic properties of the Bass--Serre dimension. Similar results for the Jacobson dimension and the property of
\( \Max(K) \) to be Noetherian can be found in
\cite{noeth}.

\begin{theorem}
	\label{bas-ser-dim}
	The Bass--Serre dimension satisfies the following.
	\begin{enumerate}
		
		\item
		\( \delta(K) \leq 0 \) if and only if
		\( K \) is semilocal;
		
		\item
		\( \delta( K / \mathfrak{a} ) \leq \delta(K) \) for any ideal
		\( \mathfrak{a} \leqt K \);
		
		\item
		if
		\( \mathfrak{a} \leqt K \) is contained in the Jacobson radical
		\( \Jac(K) \leqt K \), then
		\( \delta( K / \mathfrak{a} ) = \delta(K) \);
		
		\item
		\(
			\delta( K_1 \times K_2 )
			=
			\max( \delta(K_1), \delta(K_2) )
		\);
		
		\item
		if
		\( K \to E \) is a finite homomorphism, then
		\( \delta(K) \geq \delta(E) \);
		
	\end{enumerate}
\end{theorem}
\begin{proof}
	The space
	\( \Max( K / \mathfrak{a} ) \) is canonically homeomorphic to the closed subset
	\( \ClosedSet( \mathfrak{a} ) \subseteq \Max(K) \). Recall from
	\cite{noeth} that an ideal
	\( \mathfrak{a} \leqt K \) is called
	\( \Jac \)-radical if it coincides with
	\[
		\Jac_K(\mathfrak{a})
		=
		\bigcap
			\bigl\{
				\mathfrak{m} \in \Max(K)
				\mid
				\mathfrak{m} \geq \mathfrak{a}
			\bigr\}.
	\]
	We have
	\(
		\ClosedSet( \mathfrak{a} )
		=
		\ClosedSet( \Jac_K( \mathfrak{a} ) )
	\) and
	\( \Jac_K( \mathfrak{a} ) \) is always
	\( \Jac \)-radical, so closed subsets of
	\( \Max(K) \) are in a one-to-one order reversing correspondence with
	\( \Jac \)-radical ideals. A
	\( \Jac \)-radical ideal corresponds to an irreducible subset if and only if it is prime.
	\begin{enumerate}
		
		\item
		The ring
		\( K \) is semilocal if and only if
		\( \Max(K) \) is a finite topological space (necessarily discrete). Any finite space
		\( X \) has
		\( \delta(X) \leq 0 \). On the other hand, if a
		\( \mathrm{T}_1 \) space
		\( X \) has
		\( \delta(X) \leq 0 \), then it is a union of finitely many irreducible subspaces of zero combinatorial dimension, i.e.\ one-point subsets.
		
		\item
		If
		\( A \subseteq X \) is a subspace of a topological space, then
		\( \dim(A) \leq \dim(X) \) and
		\( \delta(A) \leq \delta(X) \).
		
		\item
		If
		\( \mathfrak{a} \leq \Jac(K) \), then
		\( \ClosedSet( \mathfrak{a} ) = \Max(K) \).
		
		\item
		If
		\( X = \bigcup_{i = 1}^n X_i \) is a finite cover of a topological space by arbitrary subspaces, then
		\( \delta(X) = \max_i \delta(X_i) \).
		
		\item
		By
		\cite[theorem 3.6]{noeth}, if
		\( f \colon K \to E \) is a finite homomorphism and
		\( X = \Max(K) \) is Noetherian, then
		\( Y = \Max(E) \) is Noetherian. Let
		\( X = \bigcup_{i = 1}^n X_i \) be a cover by subsets of combinatorial dimension
		\( \leq d \) and let
		\[
			Y_i = \Max(f)^{- 1}(X_i) \subseteq Y,
		\]
		so
		\( Y = \bigcup_{i = 1}^n Y_i \). We claim that
		\( \dim(Y_i) \leq d \). Indeed, assume that
		\( d \) is finite and there is a chain
		\[
			Z_0
			\subset
			Z_1
			\subset
			\ldots
			\subset
			Z_{d + 1}
			\subseteq
			Y_i
		\]
		of distinct closed irreducible subsets. The closure
		\( \Closure_Y(Z_j) \) of every
		\( Z_j \) is a closed irreducible subset of
		\( Y \) and these closures still form a chain of distinct subsets. The map
		\( \Max(f) \colon Y \to X \) is closed because
		\( f \) is integral, namely,
		\[
			\Max(f)( \ClosedSet( \mathfrak{a} ) )
			=
			\ClosedSet( f^{- 1}( \mathfrak{a} ) ).
		\]
		Then every
		\[
			\Max(f)\bigl( \Closure_Y(Z_i) \bigr)
			=
			\Closure_X\bigl( \Max(f)(Z_i) \bigr)
		\]
		is an irreducible closed subset of
		\( X \). Moreover,
		\( \Max(f)\bigl( \Closure_Y(Z_i) \bigr) \) are distinct because otherwise the
		\( \Jac \)-radical prime ideals corresponding to
		\( \Closure_Y(Z_i) \) lie over the same
		\( \Jac \)-radical prime ideal of
		\( K \) by
		\cite[proposition 3.1]{noeth}. It follows that
		\[
			\Closure_{X_i}\bigl( \Max(f)(Z_0) \bigr)
			\subset
			\Closure_{X_i}\bigl( \Max(f)(Z_1) \bigr)
			\subset
			\ldots
			\subset
			\Closure_{X_i}\bigl( \Max(f)(Z_{d + 1}) \bigr)
			\subseteq
			X_i
		\]
		is a chain of distinct closed irreducible subsets, a contradiction.
		\qedhere
		
	\end{enumerate}
\end{proof}

\section{Locally isotropic reductive groups}

In this paper the term
\textit{root system} means a crystallographic possibly non-reduced root system, i.e.\ it can have components of type
\( \RootSys{BC}{\ell} \) for
\( \ell \geq 1 \). We say that roots of the largest length
\( \lambda \) in a given component are
\textit{long}, roots with length strictly between
\( \lambda / 2 \) and
\( \lambda \) are
\textit{short}, and for a component of type
\( \RootSys{BC}{\ell} \) roots with length
\( \lambda / 2 \) are
\textit{ultrashort}. So in the case of
\( \RootSys{BC}{1} \) there are only two long roots and two ultrashort ones. A basis of a root system
\( \Phi \) is usually denoted by
\( \Delta \) (it consists of short and ultrashort roots for components of type
\( \RootSys{BC}{\ell} \)), and
\( \Phi^{+} \),
\( \Phi^{-} \) are the sets of positive and negative roots with respect to a basis
\( \Delta \).

Now let
\( G \) be a reductive group scheme over
\( K \) in the sense of
\cite{red-grp-sch}, i.e.\ a smooth affine group schemes such that its geometric fibers
\( G_{ \overline{ \kappa( \mathfrak{p} ) } } \) are connected reductive algebraic groups over the algebraic closures
\( \overline{ \kappa( \mathfrak{p} ) } \) of all residue fields. We call
\( G \)
\textit{simple} if it is a twisted form of a Chevalley--Demazure group scheme
\( \ChevGrp^\Lambda_K( \widetilde{\Phi}, {-} ) \) for an irreducible
\textit{absolute root system}
\( \widetilde{\Phi} \) and some choice of a weight lattice
\( \Lambda \). Over fields such group schemes are often called absolutely almost simple.

Suppose that
\( K \) is local. Then
\( G \) has a maximal split torus and all such tori are conjugate by the group
\( G(K) \)
\cite[Exp. XXVI, proposition 6.16]{red-grp-sch}. Choose such a torus
\( T \leq G \). The weight decomposition of the Lie algebra
\( \liealg \) of
\( G \) with respect to the action of
\( T \) has the form
\[
	\liealg
	=
	\liealg_0
	\oplus
	\bigoplus_{ \alpha \in \Phi }
		\liealg_\alpha,
\]
where
\( \Phi \) is a (possibly non-reduced) root system called the
\textit{relative root system} in the character lattice of
\( T \) and all weight subspaces are non-zero free
\( K \)-modules. There exists an \'etale extension
\( K \subseteq \widetilde{K} \) such that there is a maximal torus
\(
	T_{ \widetilde{K} }
	\leq
	\widetilde{T}
	\leq
	G_{ \widetilde{K} }
\) and
\( G_{ \widetilde{K} } \) has a splitting with the maximal torus
\( \widetilde{T} \). Without loss of generality
\( \widetilde{K} \) is also local, so there is the
\textit{absolute root system}
\( \widetilde{\Phi} \) in the character lattice of
\( \widetilde{T} \). The inclusion
\( T_{ \widetilde{K} } \to \widetilde{T} \) induces a map
\[
	u
	\colon
	\widetilde{\Phi} \sqcup \{ 0 \}
	\to
	\Phi \sqcup \{ 0 \}.
\]

This map
\( u \) is independent of the choices of
\( \widetilde{K} \) and
\( \widetilde{T} \) up to a linear isomorphism between the character lattices of the split tori. There is a complete description of all possible such maps
\cite{tits-ind}. It is easy to see e.g.\ following
\cite{ele-loc} that there are canonical group subschemes
\( L \leq G \) and
\( U_\alpha \leq G \) for
\( \alpha \in \Phi \) such that
\begin{align*}
		L_{ \widetilde{K} }
		&=
		\bigl\langle
			\widetilde{T},
			\widetilde{U}_\beta
			\mid
			u(\beta) = 0
		\bigr\rangle,
	&
		( U_\alpha )_{ \widetilde{K} }
		&=
		\bigl\langle
			\widetilde{U}_\beta
			\mid
			u(\beta) \in \{ \alpha, 2 \alpha \}
		\bigr\rangle
\end{align*}
as fppf sheaves, where
\( \widetilde{U}_\beta \leq G_{ \widetilde{K} } \) are root subgroups associated with
\( \widetilde{T} \). In particular,
\( \liealg_0 \) is the Lie algebra of
\( L \),
\( \liealg_\alpha \) is the Lie algebra of
\( U_\alpha \) for non-ultrashort
\( \alpha \), and
\( \liealg_\alpha \oplus \liealg_{ 2 \alpha } \) is the Lie algebra of
\( U_\alpha \) otherwise. Actually,
\( L \) is the scheme centralizer of
\( T \), as a group scheme it is reductive.

Now let
\( K \) be an arbitrary ring and
\( G \) a reductive group scheme over
\( K \). We need an analogue of the above subgroups
\( L \) and
\( U_\alpha \) to exist Zariski locally and to make this precise we give the next definition. A data
\[
	\bigl(
			u
			\colon
			\widetilde{\Phi} \sqcup \{ 0 \}
			\to
			\Phi \sqcup \{ 0 \}
		,
			\Delta
		,
			L
		,
			( U_\alpha )_{ \alpha \in \Phi }
		,
			( w_\alpha )_{ \alpha \in \Delta }
	\bigr)
\]
is an
\textit{isotropic pinning} if the following holds.
\begin{itemize}
	
	\item
	The map
	\( u \colon \widetilde{\Phi} \sqcup \{ 0 \} \to \Phi \sqcup \{ 0 \} \) is one of the possible maps
	\cite{tits-ind} from the absolute root system to the relative one. More precisely,
	\( u \) is induced by a linear map
	\( \Real \widetilde{\Phi} \to \Real \Phi \), so every component of
	\( \widetilde{\Phi} \) either lies in
	\( \Ker(u) \) or non-trivially maps to a unique component of
	\( \Phi \) by e.g.\ %
	\cite[lemma 1]{ele-loc}. Also for every component
	\( \Phi_i \subseteq \Phi \) there is at least one component
	\( \widetilde{\Phi}_j \subseteq \widetilde{\Phi} \) non-trivially mapping to
	\( \Phi_i \) and all such maps are induced by the same
	\textit{irreducible Tits index}. The complete list of these indices is well-known and can be found in
	\cite{tits-ind}. In particular,
	\(
		\widetilde{\Phi}_j \sqcup \{ 0 \}
		\to
		\Phi_i \sqcup \{ 0 \}
	\) is surjective.
	
	\item
	\( \Delta \) is a basis of
	\( \Phi \).
	
	\item
	\( L \) and
	\( U_\alpha \) for
	\( \alpha \in \Phi \) are group subschemes of
	\( G \).
	
	\item
	Fppf locally there exists a
	\textit{pinning} of
	\( G \), i.e.\ an isomorphism with the Chevalley--Demazure group scheme
	\(
		\mathbb{G}(
			\widetilde{X}^{\vee},
			\widetilde{\Phi}^{\vee};
			\widetilde{X},
			\widetilde{\Phi};
			K
		)
	\) for a root datum
	\(
		(
			\widetilde{X}^{\vee},
			\widetilde{\Phi}^{\vee};
			\widetilde{X},
			\widetilde{\Phi}
		)
	\) with given absolute root system
	\( \widetilde{\Phi} \) such that
	\begin{align*}
			L
			&=
			\bigl\langle
				\widetilde{T},
				\widetilde{U}_\beta
				\mid
				u(\beta) = 0
			\bigr\rangle,
		&
			U_\alpha
			&=
			\bigl\langle
				\widetilde{U}_\beta
				\mid
				u(\beta) \in \{ \alpha, 2 \alpha \}
			\bigr\rangle
	\end{align*}
	as fppf sheaves. Here
	\( \widetilde{T} \) is the maximal torus of
	\( G \) and
	\( \widetilde{U}_\alpha \) are the corresponding root subgroups given by the isomorphism with the Chevalley--Demazure group scheme.
	
	\item
	Elements
	\(
		w_\alpha
		\in
		U_\alpha(K)\, U_{ - \alpha } (K)\, U_\alpha(K)
	\) are
	\textit{Weyl elements}, i.e.\ %
	\( \up{w_\alpha}{U_\beta} = U_{s_\alpha(\beta)} \) and
	\( \up{w_\alpha}{L} = L \), where
	\(
		s_\alpha(\beta)
		=
		\beta
		-
		2
		\frac{
			\langle \alpha, \beta \rangle
		}{
			\langle \alpha, \alpha \rangle
		}
		\alpha
	\) is the standard reflection. If
	\( \alpha \) is the ultrashort basic root (if it exists), then we additionally require that
	\(
		w_\alpha
		\in
		U_{ 2 \alpha }(K)\,
		U_{ - 2 \alpha }(K)\,
		U_{ 2 \alpha } (K)
	\).
	
\end{itemize}
We do not distinguish isotropic pinnings with isomorphic maps
\( u \). More precisely, two isotropic pinnings
\(
	\bigl(
			u_i
			\colon
			\widetilde{\Phi_i} \sqcup \{ 0 \}
			\to
			\Phi_i \sqcup \{ 0 \}
		,
			\Delta_i
		,
			L
		,
			( U_{i \alpha} )_{ \alpha \in \Phi }
		,
			( w_{i \alpha} )_{ \alpha \in \Delta }
	\bigr)
\) for
\( i = 1, 2 \) are identified if there are isomorphisms
\( f \colon \Phi_1 \to \Phi_2 \) and
\(
	\widetilde{f}
	\colon
	\widetilde{\Phi}_1
	\to
	\widetilde{\Phi}_2
\) such that
\( u_2 \circ \widetilde{f} = f \circ u_1 \),
\( \Delta_2 = f(\Delta_1) \),
\( U_{ 1 \alpha } = U_{ 2 f(\alpha) } \), and
\( w_{ 1 \alpha } = w_{ 2 f(\alpha) } \).

The elements
\( w_\alpha \) with the required properties always exists Zariski locally by
\cite[theorem 5(3)]{weyl-ele}, at least if
\( G \) is simple (the case of adjoint
\( G \) with the Tits index
\( \TitsIndex{A}{1}{1}{1}{ (1) } \) is easy, see the example in the end of
\cite{weyl-ele}). Also, for simple
\( G \) an element
\( w \in U_\alpha(K)\, U_{ - \alpha }(K)\, U_\alpha(K) \) is Weyl in our sense if and only if
\( \up{w}{ U_\beta(K) } = U_{ s_\alpha(\beta) }(K) \) for all roots
\( \beta \) by
\cite[theorem 5(1)]{weyl-ele}, as in the standard definition (except the case of adjoint
\( G \) with the Tits index
\( \TitsIndex{A}{1}{1}{1}{ (1) } \) by the example in the end of
\cite{weyl-ele}). If
\( L \) and
\( U_\alpha \) are constructed by a maximal split torus over a local ring
\( K \), then the existence of Weyl elements follows from
\cite[Exp. XXVI, \S 7]{red-grp-sch}.

By fppf descent (see also
\cite[\S 2]{ele-loc}) root subgroups are canonically isomorphic to sheaves of modules or
\textit{%
	\( 2 \)-step nilpotent modules%
}. Namely, for every isotropic pinning there is a canonical decomposition
\[
	\liealg
	=
	\liealg_0
	\oplus
	\bigoplus_{ \alpha \in \Phi }
		\liealg_\alpha.
\]
If
\( \alpha \in \Phi \) is not ultrashort, then there is a canonical isomorphism
\[
	t_\alpha
	\colon
	\ModSheaf( \liealg_\alpha )
	\to
	U_\alpha
\]
of group schemes, where
\(
	\ModSheaf( \liealg_\alpha )(R)
	=
	\liealg_\alpha \otimes_K R
\) as a functor on the category of commutative unital
\( K \)-algebras. In the ultrashort case there is a canonical isomorphism
\[
	t_\alpha
	\colon
	\ModSheaf\bigl(
		\liealg_{ 2 \alpha } \dotoplus \liealg_\alpha
	\bigr)
	\to
	U_\alpha,
\]
where
\( \liealg_{ 2 \alpha } \dotoplus \liealg_\alpha \) is a group-theoretic central extension of
\( \liealg_\alpha \) by
\( \liealg_{ 2 \alpha } \) with an additional action of the multiplicative monoid
\( K^\bullet \) of
\( K \), and
\[
	\ModSheaf\bigl(
		\liealg_{ 2 \alpha } \dotoplus \liealg_\alpha
	\bigr)
	(R)
	=
	\bigl(
		\liealg_{ 2 \alpha } \dotoplus \liealg_\alpha
	\bigr)
	\boxtimes_K
	R
\]
is a scalar extension of
\( 2 \)-step nilpotent modules, see
\cite[\S 2]{weyl-ele} for details. In both cases we denote the domain of
\( t_\alpha \) by
\( P_\alpha \), it is an affine group scheme isomorphic to an affine space (as a scheme, i.e.\ forgetting the group structure) Zariski locally on
\( \Spec(K) \). There is the commutator formula
\[
	\bigl[ t_\alpha(x), t_\beta(y) \bigr]
	=
	\prod_{
		\substack{
				i \alpha + j \beta \in \Phi
			\\
				i, j > 0
		}
	}
		t_{ i \alpha + j \beta }\bigl(
			f_{ \alpha \beta i j }(x, y)
		\bigr)
\]
for non-antiparallel roots
\( \alpha, \beta \in \Phi \), where
\[
	f_{ \alpha \beta i j }
	\colon
	P_\alpha \times P_\beta
	\to
	P_{ i \alpha + j \beta }
\]
are homogeneous polynomial maps or their analogues for
\( 2 \)-step nilpotent modules, see
\cite[\S 2]{ele-loc} for details. The operations
\( f_{ \alpha \beta i j } \) are unique if we fix the order of factors and assume that
\( f_{ \alpha, \beta, 2 i, 2 j } = 0 \) if
\( i \alpha + j \beta, 2 i \alpha + 2 j \beta \in \Phi \).

There are several differences between our new definition of isotropic pinnings and the ones from
\cite{ele-loc, st-loc}. First of all, now we do not fix a split torus
\( T \leq L \), though there is a canonical choice of such torus at least if
\( G \) is simple
\cite[theorem 1]{weyl-ele}. Next, we do not require that the root subspaces
\( \liealg_\alpha \) are free
\( K \)-modules. This assumption is used in
\cite[\S 2]{ele-loc} only to show that
\( U_\alpha(K) \) ``splits'' as a
\( 2 \)-step nilpotent
\( K \)-module, but this is easy to check if root subspaces are mere finitely generated projective modules, see
\cite[\S 2]{weyl-ele}. The existence of Weyl elements in the definition of isotropic pinnings appears only in
\cite{st-loc} and is not a part of the original definition from
\cite{ele-loc}. Finally, the requirement that the map
\( u \) is ``standard'' (i.e.\ is given by a Tits index) is new, though such a condition is used in
\cite[\S 4.2--4.3]{st-schur}.

We do not make any restrictions on the lattice
\( \widetilde{X} \) from the root data and the embedding of
\( \widetilde{\Phi} \) to
\( \widetilde{X} \), because this is not important for our purposes. In particular, the root data may vary on
\( \Spec(K) \), and even if it is constant, we do not require that it actually appears for some maximal split torus over a local ring.

The definition of ``isotropic reductive group schemes'' from
\cite[\S 2]{weyl-ele} essentially coincides with our definition of isotropic pinnings, but without existence of Weyl elements. Also, in that paper
\( G \) is required to be simple, its root lattice constant on
\( \Spec(K) \),
\( \Phi \) is non-empty, and the ring
\( K \) is non-zero.

The next lemma shows that isotropic pinnings interact in a nice way with several standard constructions of reductive group schemes.

\begin{lemma}
	\label{iso-pin}
	Let
	\( G \) be a reductive group scheme over
	\( K \). Denote the set of isotropic pinnings on
	\( G \) by
	\( \IsoPin(G) \).
	\begin{enumerate}
		
		\item
		If
		\( C \leq G \) is a central group subscheme of multiplicative type and
		\( f \colon G \to G / C \) is the canonical homomorphism, then there is a bijection
		\begin{align*}
				\IsoPin(G) &\to \IsoPin(G / C)
			, \\
				\bigl(
					u,
					\Delta,
					L,
					( U_\alpha )_{ \alpha \in \Phi },
					( w_\alpha )_{ \alpha \in \Delta }
				\bigr)
				&\mapsto
				\bigl(
					u,
					\Delta,
					f(L),
					( f( U_\alpha ) )_\alpha,
					( f( w_\alpha ) )_\alpha
				\bigr)
			.
		\end{align*}
		
		\item
		Suppose that
		\( G = \prod_{i = 1}^n G_i \) is a direct product decomposition. Then there is an injection
		\begin{align*}
				\prod_{i = 1}^n \IsoPin(G_i)
				&\to
				\IsoPin(G)
			, \\
				\bigl(
					u_i,
					\Delta_i,
					L_i,
					( U_\alpha )_{ \alpha \in \Phi_i },
					( w_\alpha )_{
						\alpha \in \Delta \cap \Phi_i
					}
				\bigr)_{i = 1}^n
				&\mapsto
				\Bigl(
						\bigsqcup_i u_i
						\colon
						\widetilde{\Phi}
						\to
						\Phi
					,
						\Delta
					,
						\prod_i L_i
					,
						( U_\alpha )_{ \alpha \in \Phi }
					,
						( w_\alpha )_{ \alpha \in \Delta }
				\Bigr)
			,
		\end{align*}
		where
		\(
			\widetilde{\Phi}
			=
			\bigsqcup_i \widetilde{\Phi}_i
		\),
		\( \Phi = \bigsqcup_i \Phi_i \), and
		\( \Delta = \bigsqcup_i \Delta_i \). An isotropic pinning
		\(
			( u \colon \widetilde{\Phi} \to \Phi, \ldots )
			\in
			\IsoPin(G)
		\) lies in the image of this map if and only if for every component
		\( \Psi \subseteq \Phi \) all components
		\( \widetilde{\Psi} \subseteq \widetilde{\Phi} \) non-trivially mapping to
		\( \Psi \sqcup \{ 0 \} \) correspond to the same factor
		\( G_i \).
		
		\item
		Suppose that
		\( K \) is a finite \'etale algebra over a ring
		\( K_0 \) of constant degree
		\( n \), so the Weil restriction
		\( \Restrict_{ K / K_0 }(G) \) is a reductive group scheme with the absolute root system
		\( n \widetilde{\Phi} \). Then there is an injection
		\begin{align*}
				\IsoPin(G)
				&\to
				\IsoPin( \Restrict_{ K / K_0 }(G) )
			, \\
				\bigl(
					u,
					\Delta,
					L,
					( U_\alpha )_\alpha,
					( w_\alpha )_\alpha
				\bigr)
				&\mapsto
				\bigl(
					n u,
					\Delta,
					\Restrict_{ K / K_0 }(L),
					\bigl(
						\Restrict_{ K / K_0 }( U_\alpha )
					\bigr)_\alpha,
					( w_\alpha )_\alpha
				\bigr)
			,
		\end{align*}
		where
		\(
			n u
			\colon
			n \widetilde{\Phi} \sqcup \{ 0 \}
			\to
			\Phi \sqcup \{ 0 \}
		\) is the gluing of
		\( n \) copies of
		\( u \).
		
		\item
		Let
		\( K_0 \to K \) be a finite \'etale homomorphism of constant degree
		\( n \). Suppose that the absolute root system
		\( \widetilde{\Phi} \) of
		\( G \) is constant on
		\( \Spec(K) \) and
		\(
			\IsoPinning
			=
			(
					u
					\colon
					n \widetilde{\Phi} \sqcup \{ 0 \}
					\to
					\Phi \sqcup \{ 0 \}
				,
					\ldots
			)
		\) is an isotropic pinning on
		\( \Restrict_{ K / K_0 }(G) \) with indecomposable relative root system
		\( \Phi \) and no component of
		\( n \widetilde{\Phi} \) in the kernel of
		\( u \). Then
		\( \IsoPinning \) lies in the image of
		\( \IsoPin(G) \).
		
		\item
		Let
		\( K_0 \to K \) be a finite \'etale homomorphism (of possibly non-constant degree) and
		\( \IsoPinning \) an isotropic pinning on
		\( \Restrict_{ K / K_0 }(G) \). Then there are decompositions
		\( K_0 = \prod_{i = 1}^n K_{0 i} \) and
		\(
			K \otimes_{K_0} K_{0 i}
			=
			\prod_{j = 1}^{m_i} K_{i j}
		\) such that
		\( \IsoPinning|_{ K_{0 i} } \) lies in the image of the map
		\(
			\prod_{j = 1}^{m_i} \IsoPin( G|_{ K_{i j} } )
			\to
			\IsoPin( \Restrict(G)|_{ K_{0 i} } )
		\) composed from the two previous constructions.
		
	\end{enumerate}
\end{lemma}
\begin{proof}
	The first four claims are easy to check fppf locally.
	
	In order to prove the last claim we can assume that
	\( G \) is semisimple adjoint by taking the factor-group by the scheme center. Also we can assume that
	\( K \) and
	\( K_0 \) are non-zero, the absolute root system
	\( \widetilde{\Phi} \) of
	\( G \) is constant on
	\( \Spec(K) \), components of
	\( \widetilde{\Phi} \) are of the same type, and the degree
	\( n \) of
	\( K_0 \to K \) is positive constant on
	\( \Spec(K_0) \). Let
	\( G_0 = \Restrict_{ K / K_0 }(G) \), so
	\( n \widetilde{\Phi} \) is the absolute root system of
	\( G_0 \). Let also
	\(
		\IsoPinning
		=
		\bigl(
				u
				\colon
				n \widetilde{\Phi} \sqcup \{ 0 \}
				\to
				\Phi \sqcup \{ 0 \}
			,
				\Delta
			,
				L
			,
				( U_\alpha )_\alpha
			,
				( w_\alpha )_\alpha
		\bigr)
	\).
	
	If
	\( \Phi \) is empty, then there is nothing to prove. Suppose
	\( \Phi \) is non-empty reducible or irreducible and
	\( \Ker(u) \) contains some components of
	\( n \widetilde{\Phi} \). By
	\cite[Exp. XXIV, proposition 5.9]{red-grp-sch} there is a finite \'etale homomorphism
	\( K \to E \) and a simple adjoint group scheme
	\( H \) over
	\( E \) such that
	\( G \cong \Restrict_{ E / K }(H) \). Then
	\( G_0 \cong \Restrict_{ E / K_0 }(H) \). Decompositions
	\( G = G' \times G'' \) are in a one-to-one correspondence with idempotents of
	\( E \) and similarly for
	\( G_0 \). By our assumption on
	\( u \) there is a non-trivial decomposition
	\( G_0 = G'_0 \times G''_0 \) such that
	\( \IsoPinning \) lies in the image of
	\( \IsoPin(G'_0) \times \IsoPin(G''_0) \), so there exists a corresponding idempotent in
	\( E \) and a decomposition
	\( G = G' \times G'' \). Now
	\( G' \) and
	\( G'' \) have smaller relative dimension and we can prove the claim by induction.
	
	Take an fppf homomorphism
	\( K_0 \to \widetilde{K}_0 \) such that
	\( \IsoPinning_{ \widetilde{K}_0 } \) splits (i.e.\ it is obtained from a splitting of
	\( ( G_0 )_{ \widetilde{K}_0 } \) ) and the
	\( \widetilde{K}_0 \)-algebra
	\( \widetilde{K} = K \otimes_{K_0} \widetilde{K}_0 \) is isomorphic to
	\( \widetilde{K}_0^n \). Let
	\( \widetilde{G} = G|_{ \widetilde{K} } \) and
	\(
		\widetilde{G}_0
		=
		\Restrict_{ \widetilde{K} / \widetilde{K}_0 }(
			\widetilde{G}
		)
		=
		( G_0 )_{ \widetilde{K}_0 }
	\). We have
	\( \widetilde{G}_0 \cong \widetilde{H}_0^n \) for some split reductive group scheme
	\( \widetilde{H}_0 \) over
	\( \widetilde{K}_0 \).
\end{proof}

The above lemma allows us to reduce various problems about isotropic pinnings to the case of simple adjoint groups. Indeed, replacing
\( G \) by the factor-group
\( G / \Center(G) \) by its scheme center we can assume that
\( G \) is semisimple adjoint. Any given isotropic pinning on
\( G \) induces a decomposition
\( G = \prod_{i = 0}^n G_i \), where
\( G_0 \) corresponds to components from the kernel of
\( u \) and every
\( G_i \) for
\( 1 \leq i \leq n \) corresponds to components of
\( \widetilde{\Phi} \) non-trivially mapping to the
\( i \)-th component of
\( \Phi \),
\( n \) is the total number of these components. So the initial isotropic pinning reduces to isotropic pinnings on all
\( G_i \) and every such smaller isotropic pinning is either
\textit{anisotropic} (i.e.\ with
\( \Phi = \varnothing \)) or has the root system map
\( m_i u_i \colon m_i \widetilde{\Phi}_i \to \Phi_i \) for irreducible
\( \widetilde{\Phi}_i \) and
\( \Phi_i \). Finally, by
\cite[Exp. XXIV, proposition 5.9]{red-grp-sch} every group scheme
\( G_i \) for
\( i > 0 \) is the Weil restriction of a simple adjoint group scheme
\( H_i \) on a finite \'etale ring extension
\( K \subseteq K_i \) of degree
\( m_i \) with the absolute root system of type
\( \widetilde{\Phi}_i \).

As in
\cite[\S 2]{ele-loc}, we say that a reductive group scheme
\( G \) over a local ring
\( K \) has
\textit{%
	isotropic rank
	\( \ell \)%
} if either
\( \ell \) is the smallest rank of components of the relative root system
\( \Phi \) and no component of
\( \widetilde{\Phi} \) lies in the kernel of
\( u \), or such a component exists and
\( \ell = 0 \). A reductive group scheme
\( G \) over arbitrary ring
\( K \) has local isotropic rank
\( \ell \) if
\( \ell \) is the minimum of isotropic ranks of
\( G_{ K_{ \mathfrak{p} } } \) over the local rings
\( K_{ \mathfrak{p} } \) for prime ideals
\( \mathfrak{p} \leqt K \). Equivalently, one can define the
\textit{isotropic rank} of an isotropic pinning to be the smallest rank of components of
\( \Phi \) if the kernel of
\( u \) does not contain a component of
\( \widetilde{\Phi} \) and
\( 0 \) otherwise. By the next lemma
\( G \) has local isotropic rank at least
\( \ell \) if and only if Zariski locally it has an isotropic pinning of rank at least
\( \ell \).

\begin{lemma}
	\label{loc-iso-rk}
	Suppose that a reductive group scheme
	\( G \) has an isotropic pinning of rank
	\( \ell \). Then the local isotropic rank of
	\( G \) is at least
	\( \ell \).
\end{lemma}
\begin{proof}
	We only have to check that there exists a split torus
	\( T \leq G \) such that
	\( \Phi \) is the set of its non-zero weights under the action on the Lie algebra
	\( \liealg \) of
	\( G \),
	\( U_\alpha \) are its root subgroups as in
	\cite[\S 2]{ele-iso}, and
	\( L \) is its scheme centralizer. Indeed, if such a torus exists, then assuming that
	\( K \) is local there is a maximal split torus
	\( T \leq T_{\max} \leq G \) and the isotropic pinning constructed by
	\( T_{\max} \) contains the original one in the sense of
	\cite[\S 2]{ele-loc}.
	
	In order to construct a torus we can assume that
	\( G \) is semisimple simply connected,
	\( \Phi \) is irreducible, and
	\( u = n u' \colon n \widetilde{\Phi}' \to \Phi \) for an irreducible root system
	\( \widetilde{\Phi}' \). Then
	\( G = \Restrict_{K' / K}(G') \) for a reductive group scheme
	\( G' \) over a finite \'etale extension
	\( K \subseteq K' \) and the isotropic pinning on
	\( G \) is the Weil restriction of a suitable isotropic pinning on
	\( G' \) as in lemma
	\ref{iso-pin}(3). Since this new isotropic pinning has irreducible both absolute and relative root systems, it determines the canonical split torus
	\(
		T'
		\cong
		\mathbb{G}_{ \mathrm{m}, K' }^{ \mathrm{rk}(\Phi) }
	\) by
	\cite[theorem 1]{weyl-ele} and the splitting is given by suitable positive multiples of elements of
	\( \Delta \). Taking Weil restrictions we see that
	\( \Restrict_{ K' / K }(T') \) contains the required split torus of rank
	\( \mathrm{rk}(\Phi) \) over
	\( K \).
\end{proof}

Finally, local isotopic rank have the following nice properties.

\begin{lemma}
	\label{loc-iso-ope}
	Let
	\( G \) be a reductive group scheme over
	\( K \).
	\begin{enumerate}
		
		\item
		If
		\( K = \prod_{i = 1}^n K_i \), the the local isotropic rank of
		\( G \) is the minimum of such ranks of
		\( G_{K_i} \).
		
		\item
		If
		\( G = \prod_{i = 1}^n G_i \), then the local isotropic rank of
		\( G \) is the minimum of such ranks of
		\( G_i \).
		
		\item
		If
		\( C \leq G \) is a central subgroup of multiplicative type, then the local isotropic ranks of
		\( G \) and
		\( G / C \) coincide.
		
		\item
		If
		\( K_0 \to K \) is a finite \'etale homomorphism, then the local isotropic ranks of
		\( G \) and
		\( \Restrict_{K / K_0}(G) \) coincide.
		
	\end{enumerate}
\end{lemma}
\begin{proof}
	This follows from lemma
	\ref{iso-pin}.
\end{proof}

\section{Elementary groups}

Let
\( X \) be a pointed affine
\( K \)-scheme of finite presentation, i.e.\ an affine scheme of finite presentation over
\( \Spec(K) \) with a given section
\( \Spec(K) \to X \). For every non-unital commutative
\( K \)-algebra
\( A \) there is a well-defined set
\( X(A) \) of
\( A \)-points of
\( X \), namely,
\[
	X(A) = \Ker\bigl( X( A \rtimes K ) \to X(K) \bigr).
\]
Actually, this construction makes sense in a larger generality. Let
\( \mathbf{C} \) be arbitrary regular category and
\( A \) a non-unital commutative
\( K \)-algebra object in
\( \mathbf{C} \), i.e.\ an object with given morphisms
\begin{align*}
		0 &\colon 1_{ \mathbf{C} } \to A,
	\\
		( {-} ) + ( {=} ) &\colon A \times A \to A,
	\\
		- ( {-} ) &\colon A \to A,
	\\
		( {-} ) ( {=} ) &\colon A \times A \to A,
	\\
		k ( {-} ) &\colon A \to A
		\text{ for all }
		k \in K
\end{align*}
satisfying usual axioms. Then
\( X(A) \) is well defined an an object of
\( \mathbf{C} \) as the limit of a finite diagram consisting of powers of
\( A \) and morphisms definable in terms of the operations, see
\cite[\S 4]{ele-loc} for details.

Let
\( G \) be a reductive group scheme over
\( K \) with an isotropic pinning
\[
	\IsoPinning
	=
	\bigl(
		u,
		\Delta,
		L,
		( U_\alpha )_\alpha,
		( w_\alpha )_\alpha
	\bigr)
\]
of positive rank. If
\( A \) is a non-unital commutative
\( K \)-algebra in a regular category
\( \mathbf{C} \), then there are root subgroups
\( U_\alpha(A) \leq G(A) \) as group objects in
\( \mathbf{C} \), as well as isomorphisms
\( t_\alpha \colon P_\alpha(A) \to U_\alpha(A) \). If
\( \mathbf{C} \) is infinitary positive, then we have the
\textit{elementary subgroup}
\[
	\Elem_{ G, \IsoPinning }(A)
	=
	\bigl\langle
		U_\alpha(A)
		\mid
		\alpha \in \Phi
	\bigr\rangle
	\leq
	G(A),
\]
again as a group object in
\( \mathbf{C} \). See
\cite[\S 3]{st-loc} and further references therein for the necessary background in categorical logic.

An important example of not-set-theoretical
\( K \)-algebras is given by colocalization. Consider the non-unital
\( K \)-algebra
\( K^{ (s) } = \{ a^{ (s) } \mid a \in K \} \) with the operations
\begin{align*}
		a b^{ (s) } &= (a b)^{ (s) },
	&
		a^{ (s) } + b^{ (s) } &= (a + b)^{ (s) },
	&
		a^{ (s) } b^{ (s) } &= (s a b)^{ (s) },
\end{align*}
in non-associative algebra this construction is known as a
\textit{homotope} of
\( K \). Next consider the tower
\[
	\ldots
	\to
	K^{ (s^3) }
	\to
	K^{ (s^2) }
	\to
	K^{ (s) }
	\to
	K
\]
of
\( K \)-algebras, where
\( a^{ ( s^{n + 1} ) } \mapsto (s a)^{ (s^n) } \). Its formal projective limit
\(
	K^{ (s^\infty) }
	=
	\varprojlim_n^{ \Pro(\Set) }
		K^{ (s^n) }
\) is a non-unital
\( K \)-algebra in the category of pro-sets, it is called the
\textit{colocalization} of
\( K \) at
\( s \). This object is actually an algebra over
\( K_s \), i.e.\ there are individual multiplication morphisms
\(
	\frac{k}{s^n} ( {-} )
	\colon
	K^{ (s^\infty) }
	\to
	K^{ (s^\infty) }
\) for
\( \frac{k}{s^n} \in K_s \), but not the ``uniform'' multiplication
\(
	( {-} ) ( {=} )
	\colon
	K_s \times K^{ (s^\infty) }
	\to
	K^{ (s^\infty) }
\) inside
\( \Pro(\Set) \). Even if the isotropic pinning
\( \IsoPinning \) is defined only over
\( \Spec(K_s) \) (i.e.\ on the group scheme
\( G_{K_s} \)), then we still have the elementary subgroup
\(
	\Elem_{ G, \IsoPinning }\bigl( K^{ (s^\infty) } \bigr)
	\leq
	G\bigl( K^{ (s^\infty) } \bigr)
\) in the category
\( \Ind( \Pro(\Set) ) \). Here the ind-completion is required for the infinitary positivity.

Now assume that
\( G \) has local isotropic rank at least
\( 2 \), but possibly no globally defined isotropic pinnings. Choose a Zariski covering
\( \Spec(K) = \bigcup_{i = 1}^n \Spec( K_{s_i} ) \) and isotropic pinnings
\( \IsoPinning_i \) on
\( G_{ K_{s_i} } \). By
\cite[lemma 6]{ele-loc} the original ring
\( K \) is the sum of images of
\( K^{ (s_i^\infty) } \to K \), so it makes sense to define the elementary subgroup as
\[
	\Elem_G(K)
	=
	\bigl\langle
		\Image\bigl(
			\Elem_{ G, \IsoPinning_i }\bigl(
				K^{ (s_i^\infty) }
			\bigr)
			\to
			G\bigl( K^{ (s_i^\infty) } \bigr)
			\to
			G(K)
		\bigr)
		\mid
		1 \leq i \leq n
	\bigr\rangle
\]
in the category
\( \Ind( \Pro(\Set) ) \). By
\cite[lemma 19 and theorem 1]{ele-loc} the elementary subgroup actually lies in
\( \Ind(\Set) \) and it is independent of all choices.

Slightly more generally, let
\( \R \) be the forgetful functor
\( \Ring_K^{ \mathrm{fp} } \to \Set \) from the category of commutative unital
\( K \)-algebras of finite presentation, it is representable by the affine line
\( \Affine^1_K = \Spec( K[t] ) \). The functor category
\(
	\Presheaf_K
	=
	\Cat\bigl( \Ring_K^{ \mathrm{fp} }, \Set \bigr)
\) contains the category of affine
\( K \)-schemes of finite presentation as a full subcategory (being the presheaf category over it), and this functor category is infinitary positive. Clearly,
\( X(\R) \) is representable by the scheme
\( X \) for every affine
\( K \)-scheme
\( X \) of finite presentation. By
\cite[lemma 19 and theorem 1]{ele-loc} there is the elementary subgroup
\[
	\Elem_G(\R)
	=
	\bigl\langle
		\Image\bigl(
			\Elem_{ G, \IsoPinning_i }\bigl(
				\R^{ (s_i^\infty) }
			\bigr)
			\to
			G(\R)
		\bigr)
		\mid
		1 \leq i \leq n
	\bigr\rangle
	\leq
	G(\R)
\]
as an object of
\[
	\Ind( \Presheaf_K )
	\subseteq
	\Ind( \Pro( \Presheaf_K ) ),
\]
so it can also be considered as an ordinary presheaf by applying the ``evaluation'' functor
\( \Ind( \Presheaf_K ) \to \Presheaf_K \).

By
\cite[theorem 2]{ele-loc} there exists a morphism
\( t \colon P \to G \) between affine schemes of finite presentation such that
\[
	\Elem_G(\R) = \bigl\langle t( P(\R) ) \bigr\rangle.
\]
Actually, the ind-structure on this subgroup is given by the increasing filtration
\[
	1
	\subseteq
	\bigl(
		t( P(\R) ) \cup \{ 1 \} \cup t( P(\R) )^{- 1}
	\bigr)
	\subseteq
	\bigl(
		t( P(\R) ) \cup \{ 1 \} \cup t( P(\R) )^{- 1}
	\bigr)^2
	\subseteq
	\ldots
\]
and
\( P \) can be taken to be the affine space
\( \Affine_K^N \) for some positive integer
\( N \). Instead of one morphism
\( t \) from affine space one can take finitely many homomorphisms
\( t_i \colon P_i \to G \) from unipotent affine group schemes
\( P_i \). More precisely, every
\( P_i \) is isomorphic to
\( \ModSheaf(K^N) \) or
\( \ModSheaf(K^N \dotoplus K^M) \) for suitable
\( N \),
\( M \), and a
\( 2 \)-cocycle determining addition in the second case.

Now let
\( R \) be a commutative unital
\( K \)-algebra in an infinitary positive category
\( \mathbf{C} \). We define the
\textit{elementary subgroup} of
\( G(R) \) as
\[
	\Elem_G(R)
	=
	\bigl\langle
		\Image\bigl( t \colon P(R) \to G(R) \bigr)
	\bigr\rangle
	\leq
	G(R).
\]

\begin{theorem}
	\label{ele-sub-nor}
	Let
	\( G \) be a reductive group scheme over
	\( K \) of local isotropic rank at least
	\( 2 \) and
	\( R \) a commutative unital
	\( K \)-algebra in an infinitary positive category
	\( \mathbf{C} \). Then
	\( s( Q(R) ) \subseteq \Elem_G(R) \) for every morphism
	\( s \colon Q \to G \) between pointed affine schemes of finite presentation taking values in the elementary subgroup functor
	\( \Elem_G(\R) \) (considered as an object of
	\( \Presheaf_K \) or
	\( \Ind( \Presheaf_K ) \)). In particular,
	\( \Elem_G(R) \) is independent of the choice of the generating morphism
	\( t \colon P \to G \).
	
	If an affine scheme
	\( H \) of finite presentation acts on
	\( G \) by group automorphisms, i.e.\ there are morphisms
	\( a, a' \colon H \times G \to G \) such that
	\( a(h, g g') = a(h, g)\, a(h, g') \),
	\( a( h, a'(h, g) ) = g \),
	\( a'( h, a(h, g) ) = g \), then
	\[
		a\bigl( H(R), \Elem_G(R) \bigr)
		\subseteq
		\Elem_G(R).
	\]
	In particular,
	\( \Elem_G(R) \leqt G(R) \).
\end{theorem}
\begin{proof}
	Let
	\( s \colon Q \to G \) be such a morphism. By
	\cite[lemma 22]{ele-loc} it does not matter in which category we consider the elementary subgroup functor, in any case
	\[
		s( Q(\R) )
		\subseteq
		\bigl(
			t( P(\R) ) \cup \{ 1 \} \cup t( P(\R) )^{- 1}
		\bigr)^N
	\]
	for sufficiently large
	\( N \). Let
	\( A = \Gamma( Q, \mathcal{O}_Q ) \) be the affine algebra of
	\( Q \) and
	\( q_{ \mathrm{univ} } \in Q(A) \) the universal point, so
	\[
		s( q_{ \mathrm{univ} } )
		=
		\prod_{i = 1}^N
			t(p_i)^{ \varepsilon_i }
	\]
	for some signs
	\( \varepsilon_i \in \{ - 1, 0, 1 \} \) and universal elements
	\( p_i \in P(A) \). It follows that
	\[
		s(q)
		=
		\prod_{i = 1}^N
			t( p_i(q) )^{ \varepsilon_i }
	\]
	for
	\( q \colon Q(R) \) as an identity between morphisms
	\( Q(R) \to G(R) \) in the category
	\( \mathbf{C} \).
	
	Now let us prove the last claim about a morphism
	\( a \colon H \times G \to G \). By
	\cite[lemma 19]{ele-loc} the elementary subgroup
	\( \Elem_G(\R) \leqt G(\R) \) is preserved by all group scheme automorphisms of
	\( G \) inside the functor category
	\( \Presheaf_K \). In particular,
	\( H \times P \to G,\, (h, p) \mapsto a(h, t(p)) \) takes values in the elementary subgroup. By the above,
	\[
		a\bigl( H(R), t(P(R)) \bigr)
		\subseteq
		\Elem_G(R).
		\qedhere
	\]
\end{proof}

We also have the following properties.

\begin{theorem}
	\label{ele-sub-ope}
	Let
	\( G \) be a reductive group scheme over
	\( K \) of local isotropic index at least
	\( 2 \) and
	\( R \) a commutative unital
	\( K \)-algebra in an infinitary positive category
	\( \mathbf{C} \).
	\begin{enumerate}
		
		\item
		If
		\( R = \prod_{i = 1}^n R_i \) as an algebra, then
		\( \Elem_G(R) = \prod_{i = 1}^n \Elem_G(R_i) \).
		
		\item
		If
		\( G = \prod_{i = 1}^n G_i \), then
		\( \Elem_G(R) = \prod_{i = 1}^n \Elem_{G_i}(R) \).
		
		\item
		If
		\( C \leq G \) is a central group subscheme of multiplicative type and
		\( f \colon G \to G / C \) is the canonical homomorphism, then
		\( \Elem_{G / C}(R) = f( \Elem_G(R) ) \).
		
		\item
		If
		\( K_0 \to K \) is a finite \'etale homomorphism and
		\( G_0 = \Restrict_{ K / K_0 }(G) \), then
		\(
			\Elem_{G_0}(R)
			=
			\Elem_G( R \otimes_{K_0} K )
		\). Here
		\( R \otimes_{K_0} K \) means a suitable direct factor of
		\( R^n \) if
		\( K \) is a direct summand of
		\( K_0^n \) as a
		\( K_0 \)-module.
		
	\end{enumerate}
\end{theorem}
\begin{proof}
	The first property is obvious. It suffices to check the remaining ones for
	\( R = \R \in \Ind( \Pro( \Presheaf_K ) ) \), and this can be done by expanding the definition via colocalization using lemmas
	\ref{iso-pin} and
	\ref{loc-iso-ope}.
\end{proof}

\begin{theorem}
	\label{perfect}
	Let
	\( G \) be a reductive group scheme over
	\( K \) of local isotropic rank at least
	\( 2 \) and
	\( R \) a commutative unital
	\( K \)-algebra in an infinitary positive category
	\( \mathbf{C} \). Then the derived subgroup
	\( [ \Elem_G(R), \Elem_G(R) ] \) is perfect. If for every maximal ideal
	\( \mathfrak{m} \leqt K \) with the residue field
	\( \Field_2 \) the root system
	\( \widetilde{\Phi}_{ \mathfrak{m} } \) of
	\( G_{ \overline{ \kappa( \mathfrak{p} ) } } \) has no components of types
	\( \RootSys{B}{2} \) and
	\( \RootSys{G}{2} \), then
	\( \Elem_G(R) \) itself is perfect.
\end{theorem}
\begin{proof}
	Without loss of generality,
	\( R = \R \in \Ind( \Pro( \Presheaf_K ) ) \). The second claim is
	\cite[theorem 4]{ele-loc}. To prove the first claim we can instead assume that
	\( G \) is simple simply connected and
	\( R = \R^{ (s^\infty) } \), where
	\(
		G_{ K_s }
		\cong
		\ChevGrp^\simpcon_{ K_s }( \Phi, {-} )
	\) with
	\( \Phi \) of type
	\( \RootSys{B}{2} \) or
	\( \RootSys{G}{2} \), so all
	\( P_\alpha \) can be identified with
	\( \Affine^1_{ K_s } \).
	
	\begin{figure}[ht]
		\centering
		\begin{tikzpicture}
				\foreach \i in { 0, 1, ..., 3 } {
					\draw[->]
						( 0, 0 )
						--
						( \i * 90 : { 2 / sqrt(2) } );
					\draw[->]
						( 0, 0 )
						--
						( 45 + \i * 90 : 2 );
				}
				\node
					[ right ]
					at ( 0 : { 2 / sqrt(2) } )
					{
						\( \alpha \)
					};
				\node
					[ left ]
					at ( 180 : { 2 / sqrt(2) } )
					{
						\( - \alpha \)
					};
				\node
					[ below ]
					at ( 270 : { 2 / sqrt(2) } )
					{
						\( \alpha - \beta \)
					};
				\node
					[ above right, inner sep = .2em ]
					at ( 45 : 2 )
					{
						\( \beta \)
					};
				\node
					[ above left, inner sep = .2em ]
					at ( 135 : 2 )
					{
						\( \beta - 2 \alpha \)
					};
				\node
					[ below left, inner sep = .2em ]
					at ( 225 : 2 )
					{
						\( - \beta \)
					};
				\tikzset{ shift = { ( 5, 0 ) } }
				\foreach \i in { 0, 1, ..., 5 } {
					\draw[->]
						( 0, 0 )
						--
						( \i * 60 : { 2 / sqrt(3) } );
					\draw[->]
						( 0, 0 )
						--
						( 30 + \i * 60 : 2 );
				}
				\node
					[ right ]
					at ( 0 : { 2 / sqrt(3) } )
					{
						\( \alpha \)
					};
				\node
					[ above right, inner sep = .2em ]
					at ( 60 : { 2 / sqrt(3) } )
					{
						\( \beta \)
					};
				\node
					[ above left, inner sep = .2em ]
					at ( 120 : { 2 / sqrt(3) } )
					{
						\( \beta - \alpha \!\!\! \!\!\! \)
					};
				\node
					[ left ]
					at ( 180 : { 2 / sqrt(3) } )
					{
						\( - \alpha \)
					};
		\end{tikzpicture}
		\caption{%
			Root systems
			\( \RootSys{B}{2} \) and
			\( \RootSys{G}{2} \)%
		}
	\end{figure}
	
	In the first case let
	\( \Elem'_G(R) \leq G(R) \) be the group subobject generated by
	\( t_\alpha(x y)\, t_\beta(x y^2) \) and
	\( t_\beta(2 y) \) for
	\( x, y \colon R \),
	\( \angle(\alpha, \beta) = 45^\circ \),
	\( \alpha \) is short, and
	\( \beta \) is long. Clearly,
	\[
		\Elem'_G(R)
		=
		[ \Elem'_G(R), \Elem'_G(R) ]
		\leq
		[ \Elem_G(R), \Elem_G(R) ].
	\]
	On the other hand, we have
	\begin{align*}
			[ t_{ - \alpha }(x), t_\alpha(y) ]
			&=
			\up{ t_{ - \alpha }(x) }{
				\bigl[
						t_\alpha(y)\,
						t_\beta(y)
					,
						t_{ \beta - 2 \alpha }(x)
				\bigr]
			}\,
			\bigl[
					t_{ - \alpha }(x)\,
					t_{ \beta - 2 \alpha }(x)
				,
					t_\alpha(y)\,
					t_\beta(y)
			\bigr]\,
			\up{ t_\alpha(y) }{
				\bigl[
					t_\beta(y),
					t_{ - \alpha }(x)
				\bigr]
			},
		\\
			[ t_{ - \beta }(x), t_\beta(y) ]
			&=
			\up{ t_{ - \beta }(x) }{
				\bigl[
					t_\beta(y),
					t_{ \alpha - \beta }(x)
				\bigr]
			}\,
			\bigl[
					t_{ - \beta }(x)\,
					t_{ \alpha - \beta }(x)
				,
					t_\beta(y)\,
					t_\alpha(y)
			\bigr]\,
			\up{ t_\beta(y) }{
				\bigl[
						t_\alpha(y)
					,
						t_{ - \beta }(x)\,
						t_{ \alpha - \beta }(x)
				\bigr]
			}
	\end{align*}
	with
	\( \alpha \) and
	\( \beta \) as above. It easily follows that
	\( \Elem'_G(R) \) is normalized by
	\( \Elem_G(R) \) and all generators
	\( [ t_\gamma(x), t_\delta(y) ] \) of
	\( [ \Elem_G(R), \Elem_G(R) ] \) as a normal subgroup of
	\( \Elem_G(R) \) lie in
	\( \Elem'_G(R) \). Therefore
	\( \Elem'_G(R) = [ \Elem_G(R), \Elem_G(R) ] \).
	
	In the second case we similarly define
	\( \Elem'_G(R) \) as the group subobject generated by
	\( t_\alpha(2 x) \),
	\( t_\alpha(x y)\, t_\beta(x y^2) \) for short roots
	\( \alpha \) and
	\( \beta \) with
	\( \angle(\alpha, \beta) = 60^\circ \) and by
	\( t_\gamma(z) \) for long
	\( \gamma \). The corresponding calculation is
	\[
		[ t_{ - \alpha }(x), t_\alpha(y) ]
		=
		\up{ t_{ - \alpha }(x) }{
			\bigl[
					t_\alpha(y)\,
					t_\beta(y)
				,
					t_{ \beta - \alpha }(x)
			\bigr]
		}\,
		\bigl[
				t_{ - \alpha }(x)\,
				t_{ \beta - \alpha }(x)
			,
				t_\alpha(y)\,
				t_\beta(y)
		\bigr]\,
		\up{ t_\alpha(y) }{
			\bigl[
				t_\beta(y),
				t_{ - \alpha }(x)
			\bigr]
		}
	\]
	for short roots
	\( \alpha \),
	\( \beta \) at the angle
	\( 60^\circ \).
\end{proof}

\section{Relative elementary groups}

So now we have elementary subgroups
\( \Elem_G(R) \) for unital algebras
\( R \) and elementary subgroups
\( \Elem_{ G, \IsoPinning }(A) \) for non-unital algebras in the presence of an isotopic pinning
\( \IsoPinning \). The latter ones are called
\textit{unrelativized} in order to distinguish them from the relative ones introduced below.

Note that if
\( \mathbf{C} \) is an infinitary positive category, then there is a canonical functor
\( \Set \to \mathbf{C} \) sending every set to the corresponding coproduct of terminal objects. Hence a structure of
\( K \)-algebra on a ring object
\( A \) in
\( \mathbf{C} \) can be given by a multiplication morphism
\( K \times A \to A \) satisfying usual axioms. In the cases
\( \Ind( \Pro(\Set) ) \) and
\( \Ind( \Pro( \Presheaf_K ) ) \) the resulting functor is different from the usual one given via inclusions
\[
	\Set \subseteq \Pro(\Set) \subseteq \Ind( \Pro(\Set) )
	\text{ and }
	\Set
	\subseteq
	\Presheaf_K
	\subseteq
	\Pro( \Presheaf_K )
	\subseteq
	\Ind( \Pro( \Presheaf_K ) ).
\]

Let
\( A \) be a non-unital commutative algebra over a commutative unital ring object
\( R \) in an infinitary positive category
\( \mathbf{C} \). Assume further that
\( R \) itself is an algebra over
\( K \), equivalently, there is a fixed homomorphism
\( K \to R \) of unital ring objects. Then the semidirect product
\( A \rtimes R \) is also a commutative unital
\( K \)-algebra and we define the
\textit{relative elementary group} as
\[
	\Elem_G(R, A)
	=
	\Ker\bigl( \Elem_G( A \rtimes R ) \to \Elem_G(R) \bigr)
	\leq
	G(A).
\]
Here the homomorphism
\( A \rtimes R \to R \) is given by the second projection. It admits the section
\( R \to A \rtimes R \), so
\begin{align*}
		G(A \rtimes R) &= G(A) \rtimes G(R),
	&
		\Elem_G(A \rtimes R)
		&=
		\Elem_G(R, A) \rtimes \Elem_G(R).
\end{align*}

If
\( A \) itself is unital, then there is a canonical homomorphism
\( f \colon R \to A \) and the isomorphism
\[
	A \rtimes R \cong A \times R,\,
	a \rtimes r \mapsto ( a + f(r), r ).
\]
It follows that
\( \Elem_G(R, A) = \Elem_G(A) \), so this group object is independent of
\( R \).

\begin{lemma}
	\label{rel-gen}
	Let
	\( t_i \colon P_i \to G \),
	\( i \in I \) be homomorphisms from affine group schemes such that the elementary subgroup functor is generated by their images. Then
	\[
		\Elem_G(R, A)
		=
		\bigl\langle
			\up{ \Elem_G(R) }{ t_i( P_i(A) ) }
			\mid
			i \in I
		\bigr\rangle.
	\]
\end{lemma}
\begin{proof}
	Clearly, the group subobject
	\( H \) generated by all
	\( \up{ \Elem_G(R) }{ t_i( P_i(A) ) } \) is contained in
	\( \Elem_G(R, A) \). Moreover,
	\( H \) is normalized by
	\( \Elem_G(R) \) by construction and
	\[
		t_i( P_i( A \rtimes R) )
		=
		t_i( P_i(A) ) \rtimes t_i( P_i(R) )
		\subseteq
		H \rtimes \Elem_G(R).
	\]
	It follows that
	\( H \rtimes \Elem_G(R) = \Elem_G(A \rtimes R) \), so
	\( H = \Elem_G(R, A) \).
\end{proof}

We say that a non-unital commutative ring object
\( R \) in an infinitary positive category is
\textit{power idempotent} if for every
\( k \geq 1 \) the morphism
\( R \to R,\, x \mapsto x^k \) generates
\( R \) as an ideal. For example, unital rings and
\( K^{ (s) } \in \Ind( \Pro(\Set) ) \) are power idempotent.

\begin{theorem}
	\label{unrel}
	Let
	\( G \) be a reductive group scheme over
	\( K \) of local isotropic rank at least
	\( 2 \) and
	\( R \) a non-unital commutative
	\( K \)-algebra in an infinitary positive category. Assume that
	\( A \) is power idempotent. Then
	\( \Elem_G(A) = \Elem_G(R, A) \) is independent of the choice of
	\( R \) as a group subobject of
	\( G(A) \). If
	\( G \) has an isotropic pinning
	\( \IsoPinning \) of rank at least
	\( 2 \), then
	\( \Elem_G(A) = \Elem_{ G, \IsoPinning }(A) \).
\end{theorem}
\begin{proof}
	By lemma
	\ref{rel-gen} we have to check that
	\( \up{ G(R) }{ t(P(A)) } \subseteq \Elem_G(K, A) \) for any homomorphism
	\( t \colon P \to G \) of unipotent group schemes (i.e.\ of type
	\( \Affine^N \) or
	\( \Affine^N \dotoplus \Affine^M \)) taking values in the elementary subgroup functor. Moreover, we can assume that
	\( P \)  By a universal element argument this claim reduces to the particular case
	\( \up{ G(R) }{ t(P(A)) } \subseteq \Elem_G(K, A) \) in the category
	\( \Ind( \Pro(\Set) ) \), where
	\( R \) is the coordinate algebra of
	\( G \times P \times \Affine^1_K \) and
	\( A = R^{ (x^\infty) } \). The variable
	\( x \) is the coordinate on the
	\( \Affine^1_K \) factor. The rings
	\( R \) and
	\( K \) are considered as coproducts of one-element sets inside
	\( \Ind( \Pro(\Set) ) \) (except various colocalization pro-rings
	\( R^{ (s^\infty) } \) in
	\( \Pro(\Set) \)), so
	\( A \) is not an ideal in
	\( R \).
	
	Choose a Zariski covering
	\( \Spec(K) = \bigcup_{i = 1}^M \Spec( K_{s_i} ) \) with isotopic pinnings
	\( \IsoPinning_i \) on
	\( G_{ K_{s_i} } \). We have
	\( A = \sum_{i = 1}^N R^{ ( (x s_i)^\infty ) } \) by
	\cite[lemma 6]{ele-loc}. Using that
	\( P \) is unipotent and
	\( t \) is a homomorphism we get
	\[
		\up{ G(R) }{ t(P(A)) }
		\subseteq
		\prod_{i = 1}^N
			\up{ G(R) }{
				t\bigl(
					P\bigl(
						R^{ ( (x s_i)^\infty ) }
					\bigr)
				\bigr)
			}
		\subseteq
		\prod_{i = 1}^N
			\Elem_G\bigl(
				R_{s_i},
				R^{ ( (x s_i)^\infty ) }
			\bigr).
	\]
	Applying
	\cite[lemma 15]{ele-loc} we further get
	\[
		\prod_{i = 1}^N
			\Elem_G\bigl(
				R_{s_i},
				R^{ ( (x s_i)^\infty ) }
			\bigr)
		=
		\prod_{i = 1}^N
			\Elem_{ G, \IsoPinning_i }\bigl(
				R^{ ( (x s_i)^\infty ) }
			\bigr)
		\subseteq
		\Elem_G(K, A).
	\]
	
	For the second claim we just repeat the above argument with the given isotropic pinning.
\end{proof}

\section{Dimension filtration}

Let
\( G^d(K) \subseteq G(K) \) be the set of all
\( g \in G(K) \) such that the image of
\( g \) in
\( G(E) \) lies in
\( \Elem_G(E) \) for every commutative unital
\( K \)-algebra
\( E \) with
\( \delta(E) \leq d \). Here
\( d \) is a non-negative integer or
\( - \infty \). Clearly,
\( G^d(K) \leqt G(K) \) is a normal subgroup, moreover,
\( G^d(K) \) is preserved under all scheme automorphisms of the group scheme
\( G \). We obtain the
\textit{dimension filtration}
\[
	G(K)
	=
	G^{ - \infty }(K)
	\geq
	G^0(K)
	\geq
	G^1(K)
	\geq
	G^2(K)
	\geq
	\ldots
	\geq
	\Elem_G(K),
\]
it is functorial on
\( K \). If
\( \delta = \delta(K) \) is finite, then
\[
	G^\delta(K)
	=
	G^{ \delta + 1 }(K)
	=
	G^{ \delta + 2 }(K)
	=
	\ldots
	=
	\Elem_G(K),
\]
so the filtration stabilizes.

Recall that the completion of the ring
\( K \) at
\( s \in K \) is the ring
\[
	\widehat{K}_{ (s) }
	=
	\varprojlim\bigl(
		\ldots \to K / s^2 K \to K / s K \to 0
	\bigr).
\]

The
\textit{finite completion} of
\( K \) at
\( s \) is
\[
	\widetilde{K}_{ (s) }
	=
	\varinjlim_{K' \to K}
		\widehat{K'}_{ (s) },
\]
where
\( K' \) runs over all finitely generated rings with distinguished homomorphisms to
\( K \) (clearly, it suffices to consider only finitely generated subrings of
\( K \)).

\begin{lemma}
	\label{loc-com}
	Let
	\( g, h \in G(K) \) be elements such that the image of
	\( g \) in
	\( G(K_s) \) lies in
	\( \Elem_G(K_s) \) and the image of
	\( h \) in
	\( G( \widetilde{K}_{ (s) } ) \) lies in
	\( \Elem_G( \widetilde{K}_{ (s) } ) \). Then
	\( [g, h] \in \Elem_G(K) \).
\end{lemma}
\begin{proof}
	Without loss of generality,
	\( K \) is a finitely generated ring, so
	\( \widetilde{K}_{ (s) } = \widehat{K}_{ (s) } \) and
	\( K \) is Noetherian. Let
	\[
		s^\infty K
		=
		\bigcap\nolimits_n^{ \Pro(\Set) }
			s^n K
	\]
	be the image of the homomorphism
	\( K^{ (s^\infty) } \to K \), it is an ideal in the category of pro-sets. The canonical homomorphism
	\( K \to K / s^\infty K \) factors through
	\( \widehat{K}_{ (s) } \), so the image of
	\( h \) in
	\( G( K / s^\infty K ) \) lies in
	\( \Elem_G( K / s^\infty K ) \). Here all elementary subgroups are considered with the ind-structure inside the category
	\( \Ind( \Pro(\Set) ) \). Now recall that the elementary subgroup functor is generated by a morphism from an affine space, so
	\( \Elem_G(K) \to \Elem_G( K / s^\infty K ) \) is a regular epimorphism and the sequence
	\[
		G( s^\infty K ) \to G(K) \to G( K / s^\infty K )
	\]
	is left exact. It follows that
	\( h \in \Elem_G(K)\, G( s^\infty K ) \).
	
	By
	\cite[lemma 4]{ele-loc} the homomorphism
	\( K^{ (s^\infty) } \to s^\infty K \) is invertible, in particular,
	\( s^\infty K \) is an algebra over
	\( K_s \) considered as the coproduct of one-element sets in
	\( \Ind( \Pro(\Set) ) \). By theorem
	\ref{unrel} we have
	\[
		\Elem_G(K, s^\infty K)
		=
		\Elem_G(s^\infty K)
		=
		\Elem_G(K_s, s^\infty K),
	\]
	where
	\( K_s \) is considered as the coproduct of singletons in
	\( \Ind( \Pro(\Set) ) \). Using theorem
	\ref{ele-sub-nor} on normality of the elementary subgroup
	\( \Elem_G( s^\infty K \rtimes K_s ) \) we obtain
	\[
		\bigl[ g, G( s^\infty K ) \bigr]
		=
		\bigl[ \Image(g), G( s^\infty K ) \bigr]
		\subseteq
		\Elem_G( s^\infty K )
		\leq
		\Elem_G(K),
	\]
	where
	\( \Image(g) \) is the image in
	\( \Elem_G(K_s) \). Therefore
	\( [g, h] \in \Elem_G(K) \).
\end{proof}

\begin{lemma}
	\label{com-dim}
	For every commutative unital ring
	\( K \) we have
	\begin{align*}
				\widehat{K}_{ (s) }
				/
				\Jac( \widehat{K}_{ (s) } )
			&\cong
				\widetilde{K}_{ (s) }
				/
				\Jac( \widetilde{K}_{ (s) } )
			\cong
				K / \Jac(K)
			,
		&
			\delta( \widehat{K}_{ (s) } )
			&=
			\delta( \widetilde{K}_{ (s) } )
			=
			\delta( K / s K ).
	\end{align*}
\end{lemma}
\begin{proof}
	It is easy to see that
	\( s \widehat{K}_{ (s) } \) and
	\( s \widetilde{K}_{ (s) } \) are radical ideals of
	\( \widehat{K}_{ (s) } \) and
	\( \widetilde{K}_{ (s) } \), i.e.\ they are contained in the Jacobson radicals. Indeed, each
	\( 1 + s x \) is invertible with the inverse given by the convergent series
	\( 1 - s x + s^2 x^2 - s^3 x^3 + \ldots \) in the completion of
	\( K \) or its finitely generated subring. Also,
	\[
		\widehat{K}_{ (s) } / s \widehat{K}_{ (s) }
		=
		\widetilde{K}_{ (s) } / s \widetilde{K}_{ (s) }
		=
		K / s K.
	\]
	The equality between dimensions follows by theorem
	\ref{bas-ser-dim}(3).
\end{proof}

\begin{theorem}
	\label{sol-dim}
	The dimension filtration satisfies
	\[
		[ G^0(K), G^d(K) ] \leq G^{d + 1}(K)
	\]
	for
	\( d \geq 0 \).
\end{theorem}
\begin{proof}
	Without loss of generality,
	\( \delta(K) = d + 1 \). Take
	\( g \in G^0(K) \) and
	\( h \in G^d(K) \). By definition,
	\( \Max(K) = X_1 \cup \ldots \cup X_n \) for some irreducible subsets
	\( X_i \) of combinatorial dimension at most
	\( d + 1 \). Choose some maximal ideals
	\( \mathfrak{m}_i \in X_i \).
	
	Let
	\(
		S
		=
		K
		\setminus
		\bigcup_{i = 1}^n
			\mathfrak{m}_i
	\), it is a multiplicative subset and
	\( S^{- 1} K \) is semilocal. The image of
	\( g \) in
	\( G( S^{- 1} K ) \) lies in the elementary subgroup. But then there is a single element
	\( s \in S \) such that the image of
	\( g \) in
	\( G(K_s) \) lies in the elementary subgroup. Since
	\( \ClosedSet(s K) \) is a closed subset of
	\( \Max(K) \) properly intersecting every
	\( X_i \), we have
	\( \delta(K / s K) < \delta(K) \). By lemma
	\ref{com-dim} the Bass--Serre dimension of
	\( \widetilde{K}_{ (s) } \) is at most
	\( d \), so the image of
	\( h \) in
	\( G\bigl( \widetilde{K}_{ (s) } \bigr) \) lies in the elementary subgroup. Finally,
	\( [g, h] \in \Elem_G(K) \) by lemma
	\ref{loc-com}.
\end{proof}

To sum up, we reduced the problem of solvability of
\( \KFunc_1^G(K) = G(K) / \Elem_G(K) \) to the solvability of
\( G(K) / G^0(K) \). The latter group embeds to a product of set-many groups
\( \KFunc_1^G(E) \) for semilocal
\( K \)-algebras
\( E \).

Furthermore, the following lemma allows us to consider only simple group schemes with given weight lattice, e.g.\ only simply connected or adjoint ones. Recall that a
\textit{braided crossed module} consists of a group homomorphism
\( d \colon X \to H \), an action of
\( H \) on
\( X \) by group automorphisms, and a map
\( \langle {-}, {=} \rangle \colon H \times H \to X \) (a
\textit{crossed pairing}) such that
\begin{align*}
		d( \up{h}{x} ) &= \up{h}{ d(x) },
	&
		\langle h h', h'' \rangle
		&=
		\up{h}{ \langle h', h'' \rangle }\,
		\langle h, h'' \rangle,
	\\
		\up{x}{y} &= \up{ d(x) }{y},
	&
		\langle h, h' h'' \rangle
		&=
		\langle h, h' \rangle\,
		\up{h'}{ \langle h, h'' \rangle },
	\\
		d\bigl( \langle h, h' \rangle \bigr) &= [h, h'],
	&
		\langle h, d(x) \rangle &= \up{h}{x}\, x^{- 1},
	\\
		\up{h}{ \langle h', h'' \rangle }
		&=
		\bigl\langle \up{h}{h'}, \up{h}{h''} \bigr\rangle,
	&
		\langle d(x), h \rangle &= x\, ( \up{h}{x} )^{- 1}
\end{align*}
for
\( x, y \in X \) and
\( h, h', h'' \in H \).

\begin{lemma}
	\label{sol-iso}
	Let
	\( G \) be a reductive group scheme over arbitrary
	\( K \) of local isotropic rank at least
	\( 2 \). Let also
	\( C \leq G \) be a central group subscheme of multiplicative type. Then
	\[
		\KFunc_1^G(K) \to \KFunc_1^{ G / C }(K)
	\]
	is a braided crossed module in a canonical way. In particular,
	the derived subgroup
	\(
		\bigl[
			\KFunc_1^{ G / C }(K),
			\KFunc_1^{ G / C }(K)
		\bigr]
	\) contains the image of
	\( \KFunc_1^G(K) \) and the kernel of the homomorphism
	\( \KFunc_1^G(K) \to \KFunc_1^{ G / C }(K) \) is central in
	\( \KFunc_1^G(K) \).
\end{lemma}
\begin{proof}
	Clearly, if
	\( C \leq G \) is a central subgroup of any group, then
	\( G \to G / C \) is a braided crossed module in a unique way. This argument actually works for group objects in any Barr exact category, in particular, in the category of fppf sheaves on the category of
	\( K \)-schemes.
	
	Now let
	\( G \) be the group scheme from the statement. We know that
	\( G \to G / C \) is a braided crossed module of affine group schemes in a unique way. Then
	\( G(K) \to (G / C)(K) \) is canonically a braided crossed module of ordinary groups.
	
	It remains to check that all operations are well defined on the quotients
	\( \KFunc_1^G(K) \) and
	\( \KFunc_1^{ G / C }(K) \). This easy follows from generalized normality of
	\( \Elem_G(K) \leqt G(K) \) and surjectivity of
	\( \Elem_G(K) \to \Elem_{ G / C }(K) \), i.e.\ theorems
	\ref{ele-sub-nor} and
	\ref{ele-sub-ope}(3).
\end{proof}

\section{Constructions of classical groups}

Before proceeding with solvability of
\( \KFunc_1 \) over semilocal rings we need explicit construction of all classical isotropic simple group schemes, at least up to isogeny.

Let
\( S \) be a unital associative ring and
\( \lambda \in S^* \) an invertible element. A
\textit{%
	\( \lambda \)-involution%
} on
\( S \) is an additive map
\( S \to S,\, x \mapsto x^* \) such that
\( 1^* = 1 \),
\( (x y)^* = y^* x^* \) (i.e.\ the map is an anti-endomorphism),
\( x^{* *} = \lambda x \lambda^{- 1} \), and
\( \lambda^* = \lambda^{- 1} \). In this paper we need only the cases
\( \lambda = \pm 1 \), so the reader can assume that
\( \lambda \) is central. A
\textit{form parameter} is an additive subgroup
\( \Lambda \leq S \) such that
\( x^* \Lambda x \leq \Lambda \) for all
\( x \in S \) and
\[
	\{ x - x^* \lambda \mid x \in S \}
	=
	\Lambda_{\min}
	\leq
	\Lambda
	\leq
	\Lambda_{\max}
	=
	\{ x \in S \mid x + x^* \lambda = 0 \}.
\]
Clearly, both
\( \Lambda_{\min} \) and
\( \Lambda_{\max} \) are form parameters and they coincide if
\( 2 \in S^* \) (in this case there is only one form parameter). A
\textit{form ring} is a unital associative ring
\( S \) together with an element
\( \lambda \in S^* \), a
\( \lambda \)-involution, and a form parameter.

A
\textit{hermitian form} on a right module
\( M \) over a form ring
\( S \) is a biadditive map
\( B \colon M \times M \to S \) such that
\( B(m x, n y) = x^* B(m, n) y \) and
\( B(m, n) = B(n, m)^* \lambda \). Such a form is called
\textit{non-degenerate} if
\( M_S \) is finitely generated projective
\( m \mapsto B( m, {-} ) \) is a bijection between
\( M \) and the dual module
\( \Hom_S(M, S) \). A
\textit{quadratic form} is a map
\( q \colon M \to S / \Lambda \) such that
\begin{align*}
		q(m x) &= x^* q'(m) x + \Lambda,
	&
		q(m + n) &= q(m) + B(m, n) + q(m'),
	&
		B(m, m) &= q'(m) + q'(m)^* \lambda,
\end{align*}
where
\( q'(m) \in S \) denotes any fixed preimage of
\( q(m) \). The triple
\( (M, B, q) \) is called a
\textit{quadratic module}. The
\textit{unitary group} of a quadratic module
\( (M, B, q) \) is
\[
	\Unit(M)
	=
	\{
		g \in \Aut(M_S)
		\mid
		B(g m, g n) = B(m, n),
		q(g m) = q(m)
	\}.
\]

There is a general method to construct quadratic forms. Let
\( Q \colon M \times M \to S \) be a
\textit{sesquilinear form} on a right
\( S \)-module
\( M \), i.e.\ a biadditive map such that
\( Q(m x, n y) = x^* Q(m, n) y \). Then
\( (M, B, q) \) is a quadratic module, where
\begin{align*}
		B(m, n) &= Q(m, n) + Q(n, m)^* \lambda,
	&
		q(m) &= Q(m, m) + \Lambda.
\end{align*}

\begin{lemma}
	\label{qua-eve}
	Suppose that
	\( M \) is a finitely generated projective
	\( S \)-module. Then every pair of hermitian and quadratic forms on
	\( M \) can be constructed by a sesquilinear form.
\end{lemma}
\begin{proof}
	Let
	\( P \) be a finitely generated projective module such that
	\( M \oplus P \cong S^n \) is free. The module
	\( M \oplus P \) has forms
	\begin{align*}
			B( m \oplus p, m' \oplus p' ) &= B(m, n),
		&
			q( m \oplus p ) &= q(m),
	\end{align*}
	i.e.\ this is the orthogonal sum of
	\( M \) and
	\( P \), and the forms on
	\( P \) are zero. Clearly, it suffices to construct appropriate sesquilinear form on
	\( M \oplus P \), so without loss of generality
	\( M = S^n \) is free.
	
	Now let
	\( e_1, \ldots, e_n \in M \) be its basis as a free module. Both forms are completely determined by the values
	\( b_{i j} = B(e_i, e_j) \in S \) and
	\( q_i = q(e_i) \in S / \Lambda \). Moreover, these values can be arbitrary subject to the relations
	\begin{align*}
			b_{i j} &= b_{j i}^* \lambda
			\text{ for }
			i > j,
		&
			b_{i i} &= q'_i + (q'_i)^* \lambda
	\end{align*}
	where
	\( q'_i \in S \) are arbitrary preimages of
	\( q_i \). We can take
	\[
		Q\Bigl(
			\bigoplus_{i = 1}^n e_i x_i,
			\bigoplus_{i = 1}^n e_i y_i
		\Bigr)
		=
		\sum_{ 1 \leq i < j \leq n }
			x_i^* b_{i j} y_j
		+
		\sum_{i = 1}^n
			x_i^* q'_i y_i.
		\qedhere
	\]
\end{proof}

For example, let
\( P \) be a finitely generated projective module over a form ring
\( S \). The abelian group
\( \Hyp(P) = \Hom_S(P, S) \oplus P \) is a right
\( S \)-module under
\( (f \oplus p) x = x^* f \oplus p x \). Together with the forms
\begin{align*}
		B( f \oplus p, f' \oplus p' )
		&=
		f(p') + f'(p)^* \lambda,
	&
		q(f \oplus p) &= f(p) + \Lambda
\end{align*}
the module
\( \Hyp(P) \) is called a
\textit{hyperbolic quadratic module}. These forms are constructed by the sesquilinear form
\( Q( f \oplus p, f' \oplus p' ) = f(p') \), and the hermitian form is non-degenerate.

As an example of possible degenerate hermitian forms let
\( S = K \) be commutative with
\( \lambda = 1 \),
\( x^* = x \), and
\( \Lambda = 0 \). We say that
\( q \colon M \to K \) is a
\textit{traditional quadratic form} if
\( q(m x) = q(m) x^2 \) and the expression
\( B(m, n) = q(m + n) - q(m) - q(n) \) is bilinear. Clearly, all quadratic modules over this
\( S \) are precisely modules with traditional quadratic forms. The corresponding unitary group are called
\textit{orthogonal group} and denoted by
\( \Orth(q) \). If
\( M = \bigoplus_{i = 1}^{n} e_i K \) is free of finite rank, then
\( B \) is non-degenerate if and only if its
\textit{discriminant}
\( \mathrm{disc}(B) = \det_{i, j = 1}^n B(e_i, e_j) \) is invertible. In the case of odd rank
\( n = 2 \ell + 1 \) this determinant as a polynomial over
\( \Int \) in the formal variables
\( B(e_i, e_j) = B(e_j, e_i) \) for
\( i < j \) and
\( q(e_i) = \frac{1}{2} B(e_i, e_i) \) is divisible by
\( 2 \) and we call
\( \hdisc(B) = \frac{1}{2} \mathrm{disc}(B) \) the
\textit{half-discriminant} of
\( B \). A quadratic form
\( q \) on a free module of odd rank is called
\textit{semi-regular} if the corresponding half-discriminant is invertible. Both discriminant and half-discriminant preserve invertibility under base change, so we can define semi-regular quadratic forms on finitely generated projective modules of constant odd rank by Zariski localization. If
\( 2 \) is not invertible, then semi-regular quadratic forms always have degenerate symmetric bilinear forms
\( B \).

Sometimes it is more convenient to work with the graph of a quadratic form
\( q \) instead of
\( q \) itself. The
\textit{Heisenberg group} of a hermitian form
\( B \) is the set
\( \Heis(B) = M \times S \) with the group operation
\[
	(m, x) \dotplus (n, y) = ( m + n, x - B(m, n) + y ).
\]
The multiplicative monoid
\( S^\bullet\) of
\( S \) acts on
\( \Heis(B) \) by endomorphisms from the right,
\( (m, x) \cdot y = ( m y, y^* x y ) \). An
\textit{odd form parameter}
\cite{odd-uni-gro} is an
\( S^\bullet \)-invariant subgroup
\( \OddFormPar \leq \Heis(B) \) such that
\[
	\{ ( 0, x - x^* \lambda ) \mid x \in S \}
	=
	\OddFormPar_{\min}
	\leq
	\OddFormPar
	\leq
	\OddFormPar_{\max}
	=
	\{ (m, x) \mid x + B(m, m) + x^* \lambda = 0 \}.
\]
Clearly, both
\( \OddFormPar_{\min} \) and
\( \OddFormPar_{\max} \) are odd form parameters, i.e.\ they are
\( S^\bullet \)-invariant subgroups and
\( \OddFormPar_{\min} \leq \OddFormPar_{\max} \). In this case the unitary group is defined as
\[
	\Unit(M)
	=
	\{
		g \in \Aut(M_S)
		\mid
		B(g m, g n) = B(m, n),
		( g m - m, B(g m - m, m) ) \in \OddFormPar
	\}.
\]

If
\( \Lambda \) is an ordinary form parameter and
\( q \colon M \to S / \Lambda \) is a quadratic form, then
\( \OddFormPar = \{ (m, x) \mid x + \Lambda = - q(m) \} \) is an odd form parameter. Conversely, every odd form parameter
\( \OddFormPar \) such that the first projection
\( \OddFormPar \to M \) is surjective appears in this way from
\( \Lambda = \{ x \in S \mid (0, x) \in \OddFormPar \} \) and unique quadratic form.

The language of ordinary form parameters and quadratic forms allows us to easily construct orthogonal sums of modules with forms. On the other hand, the language of odd form parameters is better suited to work with elementary unitary groups.

Now let
\( (M, B, q) \) be a quadratic module over a form ring
\( S \). Take an integer
\( \ell \geq 0 \). Consider the module
\[
	e_{- \ell} S
	\oplus
	\ldots
	\oplus
	e_{- 1} S
	\oplus
	M
	\oplus
	e_1 S
	\oplus
	\ldots
	\oplus
	e_\ell S
\]
obtained from
\( M \) by adding a free direct summand of rank
\( 2 \ell \). This large module has the forms
\begin{align*}
		B\Bigl(
				\bigoplus_{i = - \ell}^{- 1}
					e_i x_i
				\oplus
				m
				\oplus
				\bigoplus_{i = 1}^{\ell}
					e_i x_i
			,
				\bigoplus_{i = - \ell}^{- 1}
					e_i y_i
				\oplus
				n
				\oplus
				\bigoplus_{i = 1}^{\ell}
					e_i y_i
		\Bigr)
		&=
		\sum_{i = - \ell}^{- 1}
			x_i^* y_{- i}
		+
		B(m, n)
		+
		\sum_{i = 1}^\ell
			x_i^* \lambda y_{- i},
	\\
		q\Bigl(
			\bigoplus_{i = - \ell}^{- 1}
				e_i x_i
			\oplus
			m
			\oplus
			\bigoplus_{i = 1}^{\ell}
				e_i x_i
		\Bigr)
		&=
		q(m)
		+
		\sum_{i = - \ell}^{- 1}
			x_i^* x_{- i}.
\end{align*}
In other words, this is the orthogonal sum of
\( M \) and the split hyperbolic quadratic module
\(
	\Hyp\bigl(
		\bigoplus_{i = 1}^\ell
			e_i S
	\bigr)
\). We denote the resulting unitary group by
\( \Unit(2 \ell, M) \). Clearly, there is a natural chain
\[
	\Unit(M)
	=
	\Unit(0, M)
	\leq
	\Unit(2, M)
	\leq
	\Unit(4, M)
	\leq
	\ldots
\]
of groups. Starting from
\( \ell = 1 \) each of these groups has the canonical elementary subgroup
\( \ElemUnit(2 \ell, M) \leq \Unit(2 \ell, M) \) generated by strictly upper and lower triangular generalized matrices from
\( \Unit(2 \ell, M) \). The factor-set is denoted by
\( \KUnit_1(2 \ell, M) \). By a special case of
\cite[theorem 4]{odd-uni-gro} the elementary subgroup is normal if
\( S \) is semilocal and
\( \ell \geq 2 \), so
\( \KUnit_1(2 \ell, M) \) is a group under these conditions. As with isotropic reductive groups, elementary subgroups of unitary groups are also generated by their
\textit{root subgroups} indexed by a root system of type
\( \RootSys{BC}{\ell} \), where ultrashort root subgroups are isomorphic to the odd form parameter
\( \OddFormPar \), short ones are isomorphic to the ring
\( S \), and the long ones are isomorphic to the ordinary form parameter
\( \Lambda \). See
\cite[\S 6]{odd-uni-gro} for details.

Finally, there is a construction of unitary groups avoiding quadratic modules. A pair
\( (R, \Delta) \) is called a (special unital) odd form ring if
\( R \) is a unital associative ring with involution
\( x \mapsto \inv{x} \) (with
\( \lambda = 1 \)) and
\( \Delta \) is an odd form parameter of the hermitian form
\( (x, y) \mapsto \inv{x} y \) on the module
\( R \). In other words,
\( \Delta \subseteq R \times R \) is closed under the operations
\begin{align*}
		(x, y) \dotplus (z, w)
		&=
		( x + z, y - \inv{x} z + w ),
	&
		(x, y) \cdot z &= ( x z, \inv{z} y z ),
\end{align*}
and
\[
	\{ ( 0, x - \inv{x} ) \mid x \in R \}
	=
	\Delta_{\min}
	\leq
	\Delta
	\leq
	\Delta_{\max}
	=
	\{
		(x, y) \in R \times R
		\mid
		y + \inv{x} x + \inv{y} = 0
	\}.
\]
The unitary group of
\( R \) considered as a quadratic module is
\[
	\Unit(R, \Delta)
	=
	\{
		g \in R^*
		\mid
		g^{- 1} = \inv{g},
		(g - 1, g - 1) \in \Delta
	\}.
\]
In
\cite{twi-for-cla} we used the condition
\( ( g - 1, \inv{g} - 1 ) \in \Delta \), but this is equivalent to
\( ( g - 1, g - 1 ) \in \Delta \) because
\( ( 0, g - \inv{g} ) \in \Delta_{\min} \).

An
\textit{orthogonal hyperbolic family} of rank
\( \ell \) in an odd form ring
\( (R, \Delta) \) is a complete family of orthogonal idempotents
\( e_{- \ell}, \ldots, e_0, \ldots, e_\ell \in R \) such that
\( \inv{e_i} = e_{- i} \) and
\( (e_i, 0) \in \Delta \) for
\( i \neq 0 \). Such a family is a replacement of a family of hyperbolic direct summand from the definition of
\( \Unit(2 \ell, M) \) above, in particular, the elementary subgroup
\( \ElemUnit(R, \Delta) \leq \Unit(R, \Delta) \) is defined as the group generated by upper and lower generalized triangular matrices from
\( \Unit(R, \Delta) \). If such a family is fixed, then root elements are defined as
\begin{align*}
		T_{i j}(x) &= 1 + x - \inv{x},
	&
		T_i(y, z) &= 1 + y + z - \inv{y},
\end{align*}
where
\( x \in e_i R e_j \),
\( y \in e_0 R e_i \),
\( z \in e_{- i} R e_i \),
\( (y, z) \in \Delta \), and
\( 0 \neq i \neq \pm j \neq 0 \). The root subgroup
\( \Image( T_{i j} ) = \Image( T_{- j, - i} ) \) corresponds to the root
\( \e_j - \e_i \), and
\( \Image( T_i ) \) corresponds to the root
\( \e_i \) of
\[
	\Phi
	=
	\{
		\pm \e_i \pm \e_j
		\mid
		1 \leq i < j \leq \ell
	\}
	\sqcup
	\{
		\pm \e_i,
		\pm 2 \e_i
		\mid
		1 \leq i \leq \ell
	\}
	\subseteq
	\Real^\ell,
\]
where
\( \e_{- i} = - \e_i \). The long root subgroup corresponding to
\( 2 \e_i \) consists of all
\( T_i(0, z) \).

Unitary groups constructed by quadratic modules and odd form algebras are related as follows. Let
\( S \) be a form ring and
\( (M, B, q) \) a quadratic module over
\( S \) with non-degenerate hermitian form
\( B \). Then
\( R = \mathrm{End}(M_S) \) has the involution
\( B( \inv{r} m, n ) = B(m, r n) \) and the form parameter
\[
	\Delta
	=
	\{
		(r, s) \in R \times R
		\mid
		s + \inv{r} r + \inv{s} = 0,
		q(r m) + B(m, s m) = 0
	\},
\]
so
\( (R, \Delta) \) is an odd form ring
\cite[\S 3]{twi-for-cla}. It is easy to see that
\( \Unit(M) = \Unit(R, \Delta) \) as subgroups of
\( R \). If we apply this construction to the quadratic module
\( M \perp \Hyp(S^\ell) \) instead of
\( S \), then the resulting endomorphism ring
\( R \) has an obvious orthogonal hyperbolic family.

Before stating main results we need to deal with triality in the case
\( \RootSys{D}{4} \). Recall that the split simple adjoint group scheme
\( \ProjSpecOrthSch_8 \) of type
\( \RootSys{D}{4} \) over
\( K \) has the automorphism group scheme
\( \ProjSpecOrthSch_8 \rtimes \Sym_3 \), where the symmetric group
\( \Sym_3 \) is considered as a constant sheaf. We say that a simple group scheme
\( G \) with absolute root system of type
\( \RootSys{D}{4} \) (or another
\textit{avoids triality} if its isomorphism class lies in the image of
\[
	\Hyp^1_\fppf(
		K,
		\ProjSpecOrthSch_8 \rtimes \Int / 2 \Int
	)
	\to
	\Hyp^1_\fppf(
		K,
		\ProjSpecOrthSch_8 \rtimes \Sym_3
	).
\]
This is independent of the choice of a subgroup
\( \Int / 2 \Int \leq \Sym_3 \) because all such subgroups are conjugate. For example,
\( G \) avoids triality if it is an inner form (from
\( \Hyp^1_\fppf( K, \ProjSpecOrthSch_8 ) \)), is a twisted form of
\( \SpecOrthSch_8 \), or has an isotropic pinning with any Tits index except
\( \TitsIndex{D}{s}{4}{1}{ (2) } \) for
\( s \in \{ 1, 2 \} \) and the indices
\( \TitsIndex{D}{s}{4}{1}{9} \),
\( \TitsIndex{D}{s}{4}{2}{2} \) for
\( s \in \{ 3, 6 \} \) explicitly involving triality.

\begin{lemma}
	\label{twi-cla}
	Let
	\( G \) be a simple adjoint group scheme over
	\( K \) with the absolute root system of classical type, i.e.\ %
	\( \RootSys{A}{\ell} \),
	\( \RootSys{B}{\ell} \),
	\( \RootSys{C}{\ell} \), or
	\( \RootSys{D}{\ell} \). In the case of
	\( \RootSys{D}{4} \) assume that
	\( G \) avoids triality. If its absolute root system is of type
	\( \RootSys{B}{\ell} \), then
	\( G \) is the orthogonal group scheme of a finitely generated projective
	\( K \)-module of constant rank
	\( 2 \ell + 1 \) with semi-regular traditional quadratic form. If the type is
	\( \RootSys{A}{\ell} \),
	\( \RootSys{C}{\ell} \), or
	\( \RootSys{D}{\ell} \), then
	\( G \) is the scheme derived subgroup of the automorphism group scheme of a sheaf
	\( ( \ModSheaf(R), \ModSheaf(\Delta) ) \) of odd form
	\( K \)-algebras,
	\( R \) is an Azumaya algebra over
	\( K \) or its quadratic \'etale extension, the first projection
	\( \Delta \to R \) is surjective, and the kernel of this projection is finitely generated projective
	\( K \)-module. The scheme derived subgroup of the unitary group sheaf is also a reductive group scheme and a finite central extension of
	\( G \).
\end{lemma}
\begin{proof}
	Recall that the group scheme
	\( \SpecOrthSch_{2 \ell + 1} \) for
	\( \ell \geq 1 \) is simple and adjoint, and its point group
	\( \SpecOrth(2 \ell + 1, K) \) is the intersection of the orthogonal group
	\( \Orth(2 \ell + 1, K) \) of the split traditional quadratic form
	\(
		q( x_{- \ell}, \ldots, x_\ell )
		=
		x_0^2 + \sum_{i = 1}^\ell x_{- i} x_i
	\) with
	\( \SpecLin(2 \ell + 1, K) \). In other words,
	\( \SpecOrth(2 \ell + 1, K) \) is the automorphism group of the triple
	\( (M, q, \omega) \), where
	\( M = K^{2 \ell + 1} \) and
	\(
		\omega
		=
		e_{- \ell} \wedge \ldots \wedge e_\ell
		\in
		\Lambda^{2 \ell + 1}(M)
	\) is the standard volume form. Since
	\( \SpecOrthSch_{2 \ell + 1} \) is its own automorphism group scheme, all its twisted forms are automorphism group schemes of triple
	\( (M, q, \omega) \), where
	\( M \) is a projective
	\( K \)-module of constant rank
	\( 2 \ell + 1 \),
	\( q \colon M \to K \) is a traditional semi-regular quadratic form on
	\( M \), and
	\( \omega \in \Lambda^{2 \ell + 1}(M) \) is a non-degenerate volume form such that
	\( \hdisc_\omega(q) = (- 1)^\ell \). Here
	\( \hdisc_\omega \) denotes the half-discriminant evaluated Zariski locally in a basis with the exterior product
	\( \omega \).
	
	Remaining cases follow from
	\cite[theorem 1]{twi-for-cla}. If
	\( G \) is a twisted form of
	\( \mathbb{PGL}_2 \cong \mathbb{PS}\mathrm{p}_2 \), then we use the symplectic odd form algebra instead of the linear one, i.e.\ a quaternion algebra with canonical involution and the maximal odd form parameter.
\end{proof}

In the next theorem the three cases are distinguished by the type of the relative root system, namely,
\( \RootSys{A}{r} \) in the first case,
\( \RootSys{B}{r} \) or
\( \RootSys{D}{r} \) in the second one, and
\( \RootSys{BC}{r} \) or
\( \RootSys{C}{r} \) in the last one.

\begin{theorem}
	\label{twi-iso-cla}
	Let
	\( G \) be a simple adjoint group scheme over
	\( K \) with the absolute root system of classical type of rank
	\( n \). Suppose that
	\( G \) has an isotropic pinning of positive rank
	\( r \geq 1 \). In the case of absolute root system
	\( \RootSys{D}{4} \) we require that the Tits index is neither
	\( \TitsIndex{D}{s}{4}{1}{ (2) } \), nor
	\( \TitsIndex{D}{s}{4}{1}{9} \), nor
	\( \TitsIndex{D}{s}{4}{2}{2} \). Also exclude the Tits indices
	\( \TitsIndex{D}{1}{4}{1}{ (4) } \) and
	\( \TitsIndex{D}{1}{4}{2}{ (2) } \), though one can instead use isomorphic indices
	\( \TitsIndex{D}{1}{4}{1}{ (1) } \) and
	\( \TitsIndex{D}{1}{4}{2}{ (1) } \). Then
	\( G \) together with the isotropic pinning has one of the following standard forms.
	\begin{itemize}
		
		\item
		For Tits indices
		\( \TitsIndex{A}{1}{n}{r}{ (d) } \) (with
		\( d (r + 1) = n + 1 \)) except
		\( r = 1 \) and
		\( d \geq 2 \) the group scheme is the automorphism group scheme of the algebra
		\( R = \Mat(\ell + 1, S) \), where
		\( S \) is an Azumaya algebra of rank
		\( d \).
		
		\item
		For Tits indices
		\( \TitsIndex{B}{}{n}{r}{} \) and
		\( \TitsIndex{D}{s}{n}{r}{ (1) } \) the group scheme is the projective special orthogonal group
		\(
			\ProjSpecOrthSch(r, M)
			=
			\ProjSpecOrthSch( M \perp \Hyp(K^r) )
		\), where
		\( M \) is finitely generated projective module of constant rank
		\( 2 n + 1 - 2 r \) with semi-regular traditional quadratic form or
		\( 2 n - 2 r \) with non-degenerate traditional quadratic form respectively.
		
		\item
		Otherwise if
		\( \TitsIndex{X}{s}{n}{r}{ (d) } \) is the Tits index, then there are a form ring
		\( S \) and a quadratic module
		\( (M, B, q) \) over
		\( S \) such that
		\( S \) is an Azumaya algeba over
		\( K \) or its quadratic \'etale extension (for
		\( \mathsf{X} = \mathsf{A} \)) of rank
		\( d \) with
		\( K \)-linear involution,
		\( \lambda = - 1 \), the form parameter
		\( \Lambda \) is a direct summand of
		\( S \) as a
		\( K \)-module, the hermitian form
		\( B \) is non-degenerate, and the rank of
		\( M \) over
		\( K \) is
		\( 2 d (n - r d) \) (for
		\( \mathsf{X} \in \{ \mathsf{C}, \mathsf{D} \} \)) or
		\( 2 d (n + 1 - 2 r d) \) (for
		\( \mathsf{X} = \mathsf{A} \)). Let
		\( \UnitSch \) be the unitary group scheme of
		\( M \perp \Hyp(S^r) \). Then the scheme derived subgroup
		\( [ \UnitSch, \UnitSch ] \) is a simple reductive group scheme and
		\( G \) is its scheme factor-group by the center.
		
	\end{itemize}
\end{theorem}
\begin{proof}
	In the first case
	\( G \) is an inner twisted form of
	\( \mathbb{PGL}_{n + 1} \), so it is the automorphism group scheme of an Azumaya algebra
	\( R \). The isotropic pinning induces a generalized matrix structure on
	\( R \), i.e.\ a complete family of orthogonal idempotents
	\( e_{0 0}, \ldots, e_{n n} \in R \). Moreover, existence of Weyl elements for basic roots easily implies that there are
	\begin{align*}
			e_{i, i + 1} &\in e_{i i} R e_{i + 1, i + 1},
		&
			e_{i + 1, i} &\in e_{i + 1, i + 1} R e_{i i}
	\end{align*}
	for
	\( 0 \leq i < n \) such that
	\begin{align*}
			e_{i, i + 1} e_{i + 1, i} &= e_{i i},
		&
			e_{i + 1, i} e_{i, i + 1} &= e_{i + 1, i + 1}.
	\end{align*}
	It follows that
	\( R \cong \Mat(n + 1, S) \) for
	\( S = e_{0 0} R e_{0 0} \).
	
	Now consider the case
	\( \RootSys{B}{}{n}{r}{} \). By lemma
	\ref{twi-cla} we have
	\( G = \SpecOrthSch(M) \) for some projective module
	\( M \) of constant rank
	\( 2 \ell + 1 \) with semi-regular traditional quadratic form
	\( q \). Existence of root subgroups induces a decomposition
	\( M = \bigoplus_{i = - r}^r M_i \), where
	\( M_i \) have rank
	\( 1 \) for
	\( i \neq 0 \),
	\( M_0 \) has rank
	\( 2 n - 2 r + 1 \),
	\( M_i \perp M_j \) unless
	\( i = - j \),
	\( q(M_i) = 0 \), and the symmetric bilinear form
	\( B \) induces isomorphisms
	\( M_{- i} \cong \Hom_K(M_i, K) = M_i^{\vee} \) for
	\( i \neq 0 \). As in the first case, existence of basic Weyl elements with long roots means that up to isomorphism
	\( M_1 = \ldots = M_r = L \) and
	\( M_{- 1} = \ldots = M_{- r} = L^{\vee} \) for a projective module
	\( L \) of constant rank
	\( 1 \). Finally, short basic root subgroups are parameterized by
	\( M_0 \otimes_K L \) and
	\( M_0 \otimes_K L^{\vee} \), and short Weyl elements exists if and only if there is
	\(
		x
		=
		\sum_{i = 1}^N
			m_i \otimes l_i
		\in
		M_0 \otimes_K L
	\) such that
	\[
		\widehat{q}(x)
		=
		\sum_{i = 1}^N
			q(m_i)\, l_i \otimes l_i
		+
		\sum_{ 1 \leq i < j \leq N }
			B(m_i, m_j)\, l_i \otimes l_j
		\in
		L \otimes_K L
	\]
	is invertible. Existence of such element means that
	\( L \otimes_K L \) has a basis, so
	\( L \cong L^{\vee} \). Finally, we can replace
	\( M \) with
	\( M \otimes_K L \) and
	\( q \) with
	\( \widehat{q} \) via the above formula to make all
	\( M_i \) free modules for
	\( i \neq 0 \). Such a replacement clearly preserves the odd orthogonal group.
	
	In the remaining cases let
	\( (R, \Delta) \) be the associated odd form ring from lemma
	\ref{twi-cla}. Below we use that it is a twisted form of one of the known ``split'' odd form algebras from
	\cite[\S 5]{twi-for-cla}, i.e.\ the ones constructed by the corresponding split quadratic modules. The set of root subgroups corresponds to an orthogonal hyperbolic family
	\( e_{- r}, \ldots, e_0, \ldots, e_r \in R \), where
	\( e_i R e_i \) is an Azumaya algebra over
	\( K \) of rank
	\( d \) for
	\( \mathsf{X} \in \{ \mathsf{C}, \mathsf{D} \} \) and an Azumaya algebra over a quadratic \'etale extension of
	\( K \) of rank
	\( d \) for
	\( \mathsf{X} = \mathsf{A} \). As in the first case, existence of Weyl elements with short roots (or long roots for
	\( \TitsIndex{D}{s}{n}{r}{ (1) } \)) means that there are
	\( e_{i j} \in e_i R e_j \) such that
	\( e_i = e_{i i} \) and
	\( e_{i j} e_{j k} = e_{i k} \) for
	\( i \),
	\( j \),
	\( k \) simultaneously positive or negative. Assume that the Tits index is
	\( \TitsIndex{D}{s}{n}{r}{ (1) } \), then
	\( e_i R e_i \cong K \) and all root elements
	\( T_i(0, x) \) are trivial, so
	\( x = \inv{x} \) for
	\( x \in e_{- i} R e_i \),
	\( i \neq 0 \). Existence of remaining basic Weyl element (both for
	\( r < d \) and for
	\( r = d \)) implies that there are
	\( e_{\mp 1, \pm 1} \in e_{\mp 1} R e_{\pm 1} \) such that
	\( e_{\pm 1, \mp 1} e_{\mp 1, \pm 1} = e_1 \). Then the module
	\( M = e_0 R e_1 \) has traditional non-degenerate quadratic form
	\( q(m) \) such that
	\( ( m, - e_{- 1, 1} q(m) ) \in \Delta \) and
	\( (R, \Delta) \) is constructed by
	\( M \perp \Hyp(K^r) \), so
	\( \Unit(R, \Delta) = \Unit(r, M) \) is an orthogonal group.
	
	For remaining Tits indices there is a basic Weyl element
	\[
		T_1( 0, e_{- 1, 1} )\,
		T_{- 1}( 0, - e_{1, - 1} )\,
		T_1( 0, e_{- 1, 1} )
	\]
	with long root, where necessarily
	\( e_{\pm 1, \mp 1} e_{\mp 1, \pm 1} = e_{\pm 1} \). Since
	\( ( 0, e_{\mp 1, \pm 1} ) \in \Delta \), we have
	\( \inv{ e_{\mp 1, \pm 1} } = - e_{\mp 1, \pm 1} \). The algebra
	\( S = e_1 R e_1 \) has the
	\( K \)-linear involution
	\( x \mapsto x^* = e_{1, - 1} \inv{x} e_{- 1, 1} \) and the form parameter
	\(
		\Lambda
		=
		\{ x \in S \mid ( 0, e_{- 1, 1} x ) \in \Delta \}
	\) with respect to
	\( \lambda = - 1 \). The
	\( S \)-module
	\( M = e_0 R e_1 \) has non-degenerate hermitian form
	\( B(m, m') = e_{1, - 1} \inv{m} m' \) and quadratic form
	\(
		q(m)
		=
		\{ x \in S \mid (m, - e_{- 1, 1} x) \in \Delta \}
	\), so the corresponding odd form parameter is
	\(
		\OddFormPar
		=
		\{
			(m, x) \in M \times S
			\mid
			(m, e_{- 1, 1} x) \in \Delta
		\}
	\). It is easy to check that
	\( (R, \Delta) \) is the odd form ring constructed by
	\( M \perp \Hyp(S^r) \), so
	\( \Unit(R, \Delta) = \Unit(r, M) \).
\end{proof}

\section{%
	Solvability of globally isotropic%
	\texorpdfstring{
		\( \KFunc_1 \)
	}{K1}%
}

In this section we prove solvability of
\( \KFunc_1^G(K) \) for semilocal
\( K \), where
\( G \) is simple with an isotropic pinning of rank at least
\( 2 \).

\begin{lemma}
	\label{lin-ab-k1}
	If
	\( S \) be a semilocal associative unital ring, then
	\( \KFunc_1(n, S) \to \KFunc_1(n + 1, S) \) is surjective for
	\( n \geq 1 \) and injective for
	\( n \geq 2 \). Moreover,
	\( \KFunc_1(2, S) \cong \KFunc_1(3, S) \cong \ldots \) is abelian.
\end{lemma}
\begin{proof}
	This is a special case of
	\cite[theorems V.4.1(b) and V.4.2]{k-theory}.
\end{proof}

\begin{lemma}
	\label{uni-ab-k1}
	Let
	\( S \) be a semilocal odd form ring and
	\( (M, B, q) \) a quadratic
	\( S \)-module. Then
	\( \KUnit_1(2 \ell, M) \to \KUnit_1(2 \ell + 2, M) \) is surjective for
	\( \ell \geq 1 \) and injective for
	\( \ell \geq 2 \). If in addition
	\( M_S \) is free of finite rank and
	\( B \) is non-degenerate, then
	\( \KUnit_1(4, M) \cong \KUnit_1(6, M) \cong \ldots \) is abelian.
\end{lemma}
\begin{proof}
	The first claim is a special case of
	\cite[theorem 5]{odd-uni-gro} and
	\cite[theorem 1.1]{bc-inj-sta}. To prove the second claim we only have to check
	\(
		[ \Unit(1, M), \Unit(1, M) ]
		\leq
		\ElemUnit(2 \ell, M)
	\)
	for sufficiently large
	\( \ell \). Consider the quadratic module
	\( M' = M \oplus M \) with the forms
	\begin{align*}
			B( m_1 \oplus m_2, n_1 \oplus n_2 )
			&=
			B(m_1, n_1) - B(m_2, n_2),
		&
			q( m_1 \oplus m_2 )
			&=
			q(m_1) - q(m_2).
	\end{align*}
	This module is hyperbolic. Namely, choose a sesquilinear form
	\( Q \colon M \times M \to S \) generating
	\( B \) and
	\( q \) using lemma
	\ref{qua-eve}. Since
	\( B \) is non-degenerate, there is a linear map
	\( f \colon M \to M \) such that
	\( B(f(m), n) = - Q(m, n) \), so
	\( B(m, f(n)) = - Q(n, m)^* \lambda \). Consider the submodules
	\begin{align*}
			N_1
			&=
			\{ m \oplus m \mid m \in M \},
		&
			N_2
			&=
			\{ (m + f(m)) \oplus f(m) \mid m \in M \}
	\end{align*}
	of
	\( M' \). Clearly, 
	\( M' = N_1 \oplus N_2 \) and the forms on
	\( M' \) are generated by the sesquilinear form
	\[
		Q'\bigl(
				(m_1 \oplus m_1)
				+
				\bigl(
					( m_2 + f(m_2) ) \oplus f(m_2)
				\bigr)
			,
				(m_1' \oplus m_1')
				+
				\bigl(
					( m_2' + f(m_2') ) \oplus f(m_2')
				\bigr)
		\bigr)
		=
		B(m_1, m_2').
	\]
	This sesquilinear form induces an isomorphism
	\( N_1 \cong \Hom_S(N_2, S) \).
	
	It follows that
	\( e_{- 1} S \oplus M \oplus e_1 S \) is a direct summand of
	\(
		\Hyp( S^{\ell - 1} )
		=
		\bigoplus_{i = - \ell}^{- 2}
			e_i S
		\oplus
		\bigoplus_{i = 2}^\ell
			e_i S
	\) for sufficiently large
	\( \ell \geq 2 \) via an embedding
	\(
		u
		\colon
		e_{- 1} S \oplus M \oplus e_1 S
		\to
		\Hyp( S^{ \ell - 1 } )
	\). By Witt cancellation theorem
	\cite[theorem 1]{ove-uni} there is
	\( h \in \Unit(2 \ell, M) \) such that
	\( h|_{ e_{- 1} S \oplus M \oplus e_1 S } = u \). Again applying stability we can assume that
	\( h \in \ElemUnit(2 \ell, M) \). Now if
	\( g_1, g_2 \in \Unit(1, M) \), then
	\[
		[g_1, g_2]
		\equiv
		[ g_1, \up{h}{g_2} ]
		=
		1
		\pmod{ \ElemUnit(2 \ell, M) }.
		\qedhere
	\]
\end{proof}

To deal with the case
\( \TitsIndex{E}{1}{6}{2}{28} \) in theorem
\ref{sol-sem-iso} below we need the following lemma about orthogonal groups of isotropic rank
\( 1 \). Let
\( M \) be a finite projective module of constant even rank over
\( K \) with a traditional non-degenerate quadratic form
\( q \colon M \to K \). Recall that
\( \SpecOrth(q) \leq \Orth(q) \) is the subgroup of elements with trivial Dickson invariant, it is a normal subgroup containing
\( [ \Orth(q), \Orth(q) ] \). A
\textit{reflection} in
\( \Orth(q) \) is an element
\( s_v \colon m \mapsto m - \frac{ B(m, v) }{ q(v) } v \), where
\( v \in M \) and
\( q(v) \in K^* \). Let
\( \Refl(q) \leq \Orth(q) \) be the subgroup generated by reflections and
\( \Refl^{+}(q) \leq \Refl(q) \cap \SpecOrth(q) \) the subgroup generated by products of pairs of reflections.

\begin{lemma}
	\label{ort-ab-k1}
	Let
	\( M \neq 0 \) be a free module of finite even rank over a semilocal ring
	\( K \) with a traditional non-degenerate quadratic form
	\( q \colon M \to K \). Then
	\( [ \Orth(q), \Orth(q) ] \leq \Refl^{+}(q) \) and the factor-group
	\( \mathrm{KO}_1(2, q) \) of
	\( \Orth( q \perp \Hyp(K) ) \) by its elementary subgroup is abelian.
\end{lemma}
\begin{proof}
	Firstly suppose that
	\( K \) is a field. By the classical Cartan--Dieudonn\'e theorem
	\( \Orth(q) = \Refl(q) \) and
	\( \SpecOrth(q) = \Refl^{+}(q) \) unless
	\( K \cong \Field_2 \) has two elements and
	\( M \cong \Hyp(K^2) \). In the exceptional case
	\(
		\Orth(q)
		\cong
		( \Sym_3 \times \Sym_3 )
		\rtimes
		\Int / 2 \Int
	\), the six reflections correspond to the elements
	\( ( \sigma, \sigma^{- 1} ) \rtimes [1] \) (the unique conjugacy class in
	\( ( \Sym_3 \times \Sym_3 ) \rtimes [1] \) with six elements), and under this isomorphism
	\begin{align*}
			\SpecOrth(q) &\cong \Sym_3 \times \Sym_3,
		\\
			\Refl(q)
			&\cong
			\{
				(\sigma, \tau) \in \Sym_3 \times \Sym_3
				\mid
				\sigma \tau
				\text{ is even}
			\}
			\rtimes
			\Int / 2 \Int,
		\\
			\Refl^{+}(q)
			&\cong
			\{
				(\sigma, \tau) \in \Sym_3 \times \Sym_3
				\mid
				\sigma \tau
				\text{ is even}
			\}.
	\end{align*}
	So for all fields
	\( [ \Orth(q), \Orth(q) ] \leq \Refl^{+}(q) \).
	
	Now return to the general case. Let
	\( \Jac(K) \leqt K \) be the Jacobson radical, so
	\( K / \Jac(K) = \prod_{i = 1}^n F_i \) is a product of fields. The map
	\[
		\Orth(q)
		\to
		\Orth( q_{ K / \Jac(K) } )
		=
		\prod_{i = 1}^n
			\Orth( q_{F_i} )
	\]
	induces the surjection
	\[
		\Refl^{+}(q)
		\to
		\prod_{i = 1}^n
			\Refl^{+}( q_{F_i} ).
	\]
	By
	\cite[theorem 4.2]{ort-sem} the kernel of
	\( \Orth(q) \to \Orth( q_{ K / \Jac(K) } ) \) is contained in
	\( \Refl^{+}(q) \), so the first claim follows from the field case.
	
	To prove the second claim note that
	\begin{align*}
			T_{+}(u)
			&=
			\Bigl(
				\begin{smallmatrix}
						1 & - B(u, {-}) & - q(u)
					\\
						0 & 1 & u
					\\
						0 & 0 & 1
				\end{smallmatrix}
			\Bigr),
		&
			T_{-}(u)
			&=
			\Bigl(
				\begin{smallmatrix}
						1 & 0 & 0
					\\
						u & 1 & 0
					\\
						- q(u) & - B(u, {-}) & 1
				\end{smallmatrix}
			\Bigr),
	\end{align*}
	for
	\( u \in M \) lie in the elementary subgroup of
	\( e_{- 1} K \oplus M \oplus e_1 K \) and
	\[
		T_{+}(v)\, T_{-}( v / q(v) )\, T_{+}(v)
		=
		\Bigl(
			\begin{smallmatrix}
					0 & 0 & q(u)
				\\
					0 & S_v & 0
				\\
					- 1 / q(u) & 0 & 0
			\end{smallmatrix}
		\Bigr).
	\]
	By the Gauss decomposition
	\( \Orth( q \perp \Hyp(K) ) \) is generated by
	\( T_{\pm}(u) \) and diagonal matrices (here we use that
	\( M \neq 0 \)), so the result follows from the first claim.
\end{proof}

Another non-trivial case in theorem
\ref{sol-sem-iso} below is
\( \TitsIndex{E}{}{7}{2}{31} \), this requires even more preliminary results. We need two classes of form algebras over
\( K \). Let us call a form ring
\( S \) a
\textit{%
	quadratic form
	\( K \)-algebra%
} if
\( S \) is a quadratic \'etale
\( K \)-algebra with the standard involution (the only geometrically non-trivial
\( K \)-linear involution),
\( \lambda = - 1 \), and
\( \Lambda = \Lambda_{\min} = \Lambda_{\max} = K \). In other words,
\( S \) is a twisted form of
\( K \times K \) with
\begin{align*}
		(x, y)^* &= (y, x),
	&
		\lambda &= - 1,
	&
		\Lambda &= \{ (x, x) \mid x \in K \}.
\end{align*}
Similarly, a form ring
\( S \) is called a
\textit{%
	quaternion form
	\( K \)-algebra%
} if
\( S \) is a quaternion
\( K \)-algebra with the standard involution
\( x^* = \mathrm{tr}(x) - x \),
\( \lambda = - 1 \), and
\( \Lambda = \Lambda_{\min} = K \). Such form algebras are precisely twisted forms of
\( \Mat(2, K) \) with
\begin{align*}
		\sMat{x}{y}{z}{w}^* &= \sMat{w}{- y}{- z}{x},
	&
		\lambda &= - 1,
	&
		\Lambda &= \{ \sMat{x}{0}{0}{x} \mid x \in K \}.
\end{align*}

\begin{lemma}
	\label{her-dia}
	Let
	\( S \) be a quadratic or quaternion form algebra over semilocal
	\( K \) and
	\( (M, B, q) \) a quadratic module over
	\( S \). Suppose that
	\( M \) is projective of constant finite rank
	\( n \) over
	\( S \) (i.e.\ of rank
	\( 2 n \) or
	\( 4 n \) over
	\( K \) respectively) and
	\( B \) is non-degenerate. Then
	\( M \) has an orthogonal basis
	\( e_1, \ldots, e_n \in M \), i.e.\ %
	\( B(e_i, e_j) = 0 \) for
	\( i \neq j \).
\end{lemma}
\begin{proof}
	In both cases
	\( S \) is a composition
	\( K \)-algebra with the standard involution,
	\( \lambda = - 1 \), and
	\( \Lambda = K \). Without loss of generality
	\( K \) is a field, so
	\( M \) has some basis
	\( e_1, \ldots, e_n \in M \). Arguing by induction on
	\( n \geq 1 \) it suffices to prove that there exist
	\( s_2, \ldots, s_n \in S \) such that
	\[
		B\Bigl(
			s_1 + \sum_{i = 2}^n e_i s_i,
			s_1 + \sum_{i = 2}^n e_i s_i
		\Bigr)
		\neq
		0.
	\]
	
	Assume the contrary. If we chose a basis of
	\( S \), then the above expression becomes a set of polynomials in coordinates of
	\( s_i \) in some basis of
	\( S \) over
	\( K \). If all of these polynomials are zero, then the hermitian form is identically zero contradicting non-degeneracy. Thus some of them is a non-zero polynomial of total degree
	\( 2 \) taking only zero values. It follows that
	\( K \cong \Field_2 \) has two elements.
	
	We have
	\begin{align*}
			B(e_1, e_1) &= 0,
		\\
			s^* B(e_i, e_j) t &= t^* B(e_i, e_j)^* s
			\text{ for }
			i > j > 1,
			\tag{%
				\( * \)%
			}
		\\
			s^* B(e_i, e_i) s
			&=
			B(e_1, e_i) s + s^* B(e_1, e_i)^*
			\text{ for }
			i > 1.
			\tag{%
				\( * * \)%
			}
	\end{align*}
	Taking
	\( s = t = 1 \) in (%
		\( * \)%
	) we get
	\( B(e_i, e_j)^* = B(e_i, e_j) \), and next taking only
	\( t = 1 \) we obtain
	\( B(e_i, e_j) s = s^* B(e_i, e_j) \). In the quaternion case this means that
	\(
		B(e_i, e_j) s t
		=
		t^* s^* B(e_i, e_j)
		=
		B(e_i, e_j) t s
	\), but additive commutators
	\( [t, s] \) generate
	\( S \) as a right ideal, so
	\( B(e_i, e_j) = 0 \). In the quadratic case the trace map
	\( S \to K,\, s \mapsto s + s^* \) is surjective and again
	\( B(e_i, e_j) = 0 \).
	
	Now consider (%
		\( * * \)%
	). Linearizing this identity we get
	\( s^* B(e_i, e_i) t = t^* B(e_i, e_i) s \), so
	\( B(e_i, e_i) = 0 \) by the above argument. But then (%
		\( * * \)%
	) reduces to
	\( B(e_1, e_i) s + s^* B(e_1, e_i)^* \) and again
	\( B(e_1, e_i) = 0 \). So
	\( B \) is identically zero, a contradiction.
\end{proof}

The next result is proved in a more general form in
\cite[theorem 4.6]{uni-sem}, but assuming that
\( S \) has no residue fields with
\( 2 \) elements. We give a simpler proof only for quadratic form algebras. Recall that is
\( M \) is a quadratic module over a quadratic form algebra
\( S \) such that
\( M_S \) is free of finite rank and the hermitian form is non-degenerate, then
\( \SpecUnit(M) = \Unit(M) \cap \SpecLin(M) \), i.e.\ the group of unitary matrices with the determinant
\( 1 \in S \). As for unitary and elementary groups,
\( \SpecUnit(n, M) = \SpecUnit( M \perp \Hyp(S^n) ) \).

\begin{lemma}
	\label{2a-ab-k1}
	Let
	\( S \) be a quadratic form algebra over semilocal
	\( K \) and
	\( M \) a quadratic module over
	\( S \). Suppose that the hermitian form is non-degenerate and the rank of
	\( M \) over
	\( K \) is constant. Then
	\( \ElemUnit(1, M) = \SpecUnit(1, M) \) and
	\( \KUnit_1(1, M) \) is abelian.
\end{lemma}
\begin{proof}
	By lemma
	\ref{her-dia} the module
	\( M \) has an orthogonal basis
	\( e_1, \ldots, e_n \in M \). We prove that arbitrary
	\( g \in \Unit(1, M) \) can be multiplied by elementary transformations from the right to get an isometry stabilizing all
	\( e_i \). Recall that root elements have the form
	\begin{align*}
			T_{+}(m, x)
			&=
			\Bigl(
				\begin{smallmatrix}
						1 & B( m, {-} ) & x
					\\
						0 & 1 & m
					\\
						0 & 0 & 1
				\end{smallmatrix}
			\Bigr),
		&
			T_{-}(m, x)
			&=
			\Bigl(
				\begin{smallmatrix}
						1 & 0 & 0
					\\
						m & 1 & 0
					\\
						- x & - B( m, {-} ) & 1
				\end{smallmatrix}
			\Bigr)
	\end{align*}
	as endomorphisms of
	\( e_{-} S \oplus M \oplus e_{+} S \), where
	\( m \in M \),
	\( x \in S \), and
	\( q(m) = x + K \).
	
	Assume that
	\( g e_i = e_i \) for
	\( i < k \), so
	\( B(g e_k, e_i) = 0 \) for
	\( i < k \). Let
	\( g e_k = e_{-} a \oplus m \oplus e_{+} b \). If
	\( a \in S^* \), then
	\[
		T_{+}(u, q'(u))\,
		T_{-}\bigl(
			(e_k - m) a^{- 1},
			B(m - e_k, m) (a a^*)^{- 1} + b a^{- 1}
		\bigr)\,
		g
	\]
	stabilizes
	\( e_i \) for
	\( i \leq k \), where
	\( u \in M \) is a vector such that
	\( B(u, e_k) = - a \) and
	\( q'(u) \) is a lift of
	\( q(u) \) to
	\( S \). So the claim about
	\( g \) follows by induction.
	
	Now suppose that
	\( g \in \SpecUnit(1, M) \) stabilizes
	\( M \) pointwise. The group of such elements is isomorphic to
	\( \SpecUnit( \Hyp(S) ) \cong \SpecLin(2, K) \) and this isomorphic preserves both root subgroups. The result from the statement now follows from
	\( \SpecLin(2, K) = \Elem(2, K) \).
\end{proof}

\begin{theorem}
	\label{sol-sem-iso}
	Let
	\( G \) be a simple group scheme over a semilocal ring
	\( K \) with an isotropic pinning of rank at least
	\( 2 \). If the Tits index is
	\( \TitsIndex{E}{}{8}{2}{78} \) assume that
	\( K \) contains a field. Then the group
	\( \KFunc_1^G(K) \) is solvable with solvability length bounded by an absolute constant.
\end{theorem}
\begin{proof}
	We freely use lemma
	\ref{sol-iso} to replace
	\( G \) by an isogenic group scheme if necessary. Classical Tits indices are covered by lemmas
	\ref{lin-ab-k1} and
	\ref{uni-ab-k1} using theorem
	\ref{twi-iso-cla}. By the Gauss decomposition
	\( G(K) \) is generated by
	\( \Elem_G(K) \) and
	\( L(K) \), so the claim trivially holds if
	\( L \) is commutative. This holds if
	\( G \) is quasi-split by its isotropic pinnings, i.e.\ for the Tits indices
	\[
		\TitsIndex{D}{s}{4}{2}{2},\,
		\TitsIndex{E}{1}{6}{6}{0},\,
		\TitsIndex{E}{2}{6}{4}{2},\,
		\TitsIndex{E}{}{7}{7}{0},\,
		\TitsIndex{E}{}{8}{8}{0},\,
		\TitsIndex{F}{}{4}{4}{0},\,
		\TitsIndex{G}{}{2}{2}{0}.
	\]
	The case
	\( \TitsIndex{E}{}{8}{2}{78} \) is proven in
	\cite[corollary 6.11(10)]{k1-red-tri}.
	
	In the cases
	\( \TitsIndex{E}{1}{6}{2}{16} \),
	\( \TitsIndex{E}{2}{6}{2}{16''} \) the relative root system has type
	\( \RootSys{G}{2} \) and the absolute root system of
	\( L \) has type
	\( 2 \RootSys{A}{2} \). Let
	\( H \leq G \) be the closed reductive group subscheme generated by
	\( U_{- \alpha} \),
	\( [L, L] \),
	\( U_\alpha \) as an fppf sheaf for a short root
	\( \alpha \). Considering the root diagram of
	\( \RootSys{E}{6} \) from
	\cite{atlas} one can see that
	\( H \) is simple and
	\( U_{ \pm \alpha } \) determine an isotropic pinning with Tits index
	\( \TitsIndex{A}{1}{5}{1}{ (3) } \). We are done by theorem
	\ref{twi-iso-cla} and lemma
	\ref{lin-ab-k1}.
	
	Next consider
	\( \TitsIndex{E}{2}{6}{2}{16'} \) and
	\( \TitsIndex{E}{}{8}{2}{66} \). In both these cases the relative root system has type
	\( \RootSys{BC}{2} \) and the scheme derived subgroup
	\( [L, L] \) is simple. Let
	\( H \leq G \) be the closed reductive group subscheme generated by root subgroups with roots from
	\( \RootSys{C}{2} \subseteq \RootSys{BC}{2} \), it contains
	\( [L, L] \). The root diagrams from
	\cite{atlas} show that the Tits indices of
	\( H \) are
	\( \TitsIndex{D}{2}{5}{2}{ (1) } \) and
	\( \TitsIndex{D}{2}{8}{2}{ (1) } \), so the result follows from theorem
	\ref{twi-iso-cla} and lemma
	\ref{uni-ab-k1}.
	
	For
	\( \TitsIndex{E}{}{7}{4}{9} \) the relative root system has type
	\( \RootSys{F}{4} \) and the absolute root system of
	\( L \) has type
	\( 3 \RootSys{A}{1} \). Let
	\( H \leq G \) be the closed reductive group subscheme generated by root subgroups with roots from
	\( \RootSys{C}{3} \subseteq \RootSys{F}{4} \). Looking and the root diagram
	\cite{atlas} one can see that the corresponding isotropic pinning on
	\( H \) has Tits index
	\( \TitsIndex{D}{1}{6}{3}{ (2) } \) and
	\( [L, L] \leq H \), so
	\( \KFunc_1^G(K) \) is again solvable by theorem
	\ref{twi-iso-cla} and lemma
	\ref{uni-ab-k1}.
	
	Now consider the Tits indices
	\( \TitsIndex{E}{1}{6}{2}{28} \),
	\( \TitsIndex{E}{}{7}{3}{28} \), and
	\( \TitsIndex{E}{}{8}{4}{28} \). In these cases
	\( [L, L] \) is simple with the absolute root system of type
	\( \RootSys{D}{4} \). The relative root system of
	\( G \) has types
	\( \RootSys{A}{2} \),
	\( \RootSys{C}{3} \), and
	\( \RootSys{F}{4} \) respectively. Let
	\( H \leq G \) be the closed reductive group subscheme generated by
	\( U_{- \alpha} \),
	\( [L, L] \),
	\( U_\alpha \) as an fppf sheaf. By
	\cite{atlas} this group scheme is simple with isotropic pinning of Tits index
	\( \TitsIndex{D}{2}{5}{1}{ (1) } \) and now we apply theorem
	\ref{twi-iso-cla} and lemma
	\ref{ort-ab-k1}.
	
	In the remaining case
	\( \TitsIndex{E}{}{7}{2}{31} \) the relative root system has type
	\( \RootSys{BC}{2} \) and the absolute root system of
	\( L \) has type
	\( \RootSys{A}{1} + \RootSys{D}{4} \). Firstly let
	\( H_1 \leq G \) be the closed reductive group subscheme generated by root subgroups with roots from the relative root subsystem of type
	\( \RootSys{C}{2} \). By
	\cite{atlas} it has isotropic pinning with the Tits index
	\( \TitsIndex{D}{2}{6}{2}{ (1) } \), so theorem
	\ref{twi-iso-cla} and lemma
	\ref{uni-ab-k1} imply that
	\( (L \cap H_1)(K) \) is solvable modulo its intersection with
	\( \Elem_G(K) \). The group scheme
	\( [L, L] \) is a central product of two simple group subschemes
	\( L_1 \) and
	\( L_2 \), the first one with the absolute root system
	\( \RootSys{D}{4} \) and the second one with the absolute root system
	\( \RootSys{A}{1} \), and
	\( L \cap H_1 \) contains only
	\( L_1 \).
	
	To deal with the second factor consider another closed reductive group subscheme
	\( H_2 \leq G \) generated by root subgroups with roots from a root subsystem of type
	\( \RootSys{BC}{1} \). By
	\cite{atlas} it has isotropic pinning with the Tits index
	\( \TitsIndex{D}{2}{6}{1}{ (2) } \) and it contains the whole group subscheme
	\( [L, L] \). It remains to prove that
	\( H_2(K) / ( \Elem_{H_2}(K)\, L_1 ) \) is solvable. Again by theorem
	\ref{twi-iso-cla} the group scheme
	\( H_2 \) up to isogeny is the connected component of the unitary group scheme of a quadratic module
	\( M \perp \Hyp(S) \) over a quaternion form
	\( K \)-algebra
	\( S \), where the hermitian form is non-degenerate and
	\( M \) has constant rank
	\( 16 \) over
	\( K \). By lemma
	\ref{her-dia} this module
	\( M \) has an orthogonal basis
	\( (e_1, e_2, e_3, e_4) \). Let
	\( H_3 \leq H_2 \) be the closed reductive group subscheme corresponding to the connected component of the unitary group scheme of the quadratic module
	\( e_1 S \perp \Hyp(S) \), it still contains the factor
	\( L_2 \), but its Tits index is
	\( \TitsIndex{D}{2}{3}{1}{ (2) } \).
	
	This Tits index is equivalent to
	\( \TitsIndex{A}{2}{3}{1}{ (1) } \). In other words, up to isogeny
	\( H_3 \) is also the special unitary group scheme of a quadratic module
	\( M' \perp \Hyp(S') \) over a quadratic form
	\( K \)-algebra
	\( S' \) by the final application of theorem
	\ref{twi-iso-cla}, where the hermitian form is again non-degenerate and
	\( M' \) has constant rank
	\( 4 \) over
	\( K \). The result follows from lemma
	\ref{2a-ab-k1}.
\end{proof}

\section{%
	Solvability of anisotropic%
	\texorpdfstring{
		\( \KFunc_1 \)
	}{K1}%
}

In this section the reader can assume that
\( G \) is a simple group scheme of local isotopic rank at least
\( 2 \) over semilocal ring
\( K \). In order to prove solvability of
\( \KFunc_1^G(K) \) let
\( G^{ [n] }(K) \subseteq G(K) \) be the set of all
\( G(K) \) such that the image of
\( G \) in
\( G(E) \) lies in
\( \Elem_G(E) \) for every commutative unital
\( K \)-algebra
\( E \) such that
\( E \) is semilocal with at most
\( n \) maximal ideals. Here
\( n \geq 0 \) is an integer. As in the case of dimension filtration
\( G^{ [n] }(K) \leqt G(K) \) is a normal subgroup preserved under all scheme automorphisms of
\( G \) as a group scheme. We obtain the
\textit{cardinality filtration}
\[
	G(K)
	=
	G^{ [0] }(K)
	\geq
	G^{ [1] }(K)
	\geq
	\ldots
	\geq
	\Elem_G(K).
\]
This filtration always stabilizes at
\( G^{ [N] }(K) = \Elem_G(K) \), where
\( N = | { \Max(K) } | \). The group
\( G(K) / G^{ [1] }(K) \) embeds into a product of set-many
\( \KFunc_1^G(E) \) for local
\( K \)-algebras
\( E \), its solvability is already considered in theorem
\ref{sol-sem-iso}.

\begin{theorem}
	\label{sol-card}
	Let
	\( K \) be a semilocal commutative unital ring and
	\( G \) a simple group scheme over
	\( K \) of local isotropic rank at least
	\( 2 \). Then the cardinality filtration satisfies
	\[
		\bigl[ G^{ [1] }(K), G^{ [n] }(K) \bigr]
		\leq
		G^{ [n + 1] }(K).
	\]
\end{theorem}
\begin{proof}
	Without loss of generality,
	\( K \) has precisely
	\( n + 1 \) maximal ideals. Take
	\( g \in G^{ [1] }(K) \),
	\( h \in G^{ [n] }(K) \), and
	\( \mathfrak{m} \in \Max(K) \). By definition, the image of
	\( g \) in
	\( G( K_{ \mathfrak{m} } ) \) lies in
	\( \Elem_G( K_{ \mathfrak{m} } ) \). It follows that there is
	\( s \in K \setminus \mathfrak{m} \) such that the image of
	\( g \) in
	\( G(K_s) \) lies in the elementary subgroup. The closed subset
	\( \ClosedSet( s K ) \subseteq \Max(K) \) is proper, so the semilocal rings
	\( K / s K \) and
	\( \widetilde{K}_{ (s) } \) have at most
	\( n \) maximal ideals by lemma
	\ref{com-dim}. Then the image of
	\( h \) in
	\( G( \widetilde{K}_{ (s) } ) \) lies in the elementary subgroup and
	\( [g, h] \in \Elem_G(K) \) by lemma
	\ref{loc-com}.
\end{proof}

The above result proves that
\( G^{ [1] }(K) / \Elem_G(K) \) is solvable, but the solvability length depends on the number of maximal ideals. For classical simple group schemes we can actually prove a uniform bound on the solvability length.

A non-unital associative ring object
\( A \) in an regular category
\( \mathbf{C} \) is called
\textit{radical} if for every
\( x \colon A \) exists (necessarily unique)
\( y \colon A \) such that
\( x + y + x y = x + y + y x = 0 \). More formally, the formula
\[
	\exists y \colon A \enskip
		x + y + x y = x + y + y x = 0
\] in regular logic holds for
\( x \colon A \). It is easy to see that a two-sided ideal
\( A \leqt R \) of a unital associative ring object is radical if and only if
\( 1 + A \subseteq R^* \). Indeed, if
\( 1 + A \subseteq R^* \) and
\( x \colon A \), then
\[
	y = (1 + x)^{- 1} - 1 = - x (1 + x)^{- 1} \colon A.
\]
It follows by taking
\( R = A \rtimes \Int \) that two-sided ideals of radical ring objects are also radical.

The following lemma is not used later. We include it because of simplicity and as a motivation for the whole idea of decomposing
\( G(A) \). Below we prove lemma
\ref{spl-uni} actually used in the proof of the main theorem. Its claim is a bit weaker for linear groups, but covers all classical groups simultaneously.

\begin{lemma}
	\label{spl-lin}
	Let
	\( K \) be arbitrary commutative unital ring,
	\( R \) an Azumaya algebra, and
	\( G \) the group scheme of invertible elements of
	\( R \). Let also
	\( A \) be a radical commutative
	\( K \)-algebra in a regular category and assume that
	\( A = \sum_{i = 1}^n A_i \) is a sum of ideals. Then
	\( G(A) = \prod_{i = 1}^n G(A_i) \).
\end{lemma}
\begin{proof}
	Note that all sums
	\( \sum_{i \in I} A_i \) are radical, where
	\( I \subseteq \{ 1, \ldots, n \} \). Thus it suffices to prove the lemma only for
	\( n = 2 \). The ring objects
	\( R \otimes_K A \) and
	\( R \otimes_K A_i \) are also radical because
	\( x \colon R \otimes_K (A \rtimes K) \) is invertible if and only if
	\( \mathrm{n}(x) \colon A \rtimes K \) is invertible, where
	\( \mathrm{n} \colon R \to K \) is the reduced norm.
	
	Now take
	\[
		g
		\colon
		G(A)
		=
		1 + R \otimes_K A
		\subseteq
		R \otimes_K (A \rtimes K),
	\]
	so
	\( g = 1 + a_1 + a_2 \) for
	\( a_i \colon R \otimes_K A_i \). Let also
	\( b_1 \) be the quasi-inverse to
	\( a_1 \), i.e.\ %
	\(
		1 + b_1
		=
		(1 + a_1)^{- 1}
		\colon
		1 + R \otimes_K A_1
	\). Then
	\[
		g = (1 + a_1) (1 + a_2 + b_1 a_2)
	\]
	and
	\( a_2 + b_1 a_2 \colon R \otimes_K A_2 \).
\end{proof}

Other classical groups schemes require a bit more work.

\begin{lemma}
	\label{big-cell}
	Let
	\( G \) be a reductive group scheme over arbitrary commutative unital ring
	\( K \) with an isotropic pinning. Let also
	\( A \) be a radical commutative
	\( K \)-algebra in a regular category. Then
	\( G(A) = U^{-}(A)\, L(A)\, U^{+}(A) \), where
	\(
		U^{\pm}
		=
		\prod_{
			\alpha \in \Phi^{\pm} \setminus 2 \Phi^{\pm}
		}
			U_\alpha
	\) are the unipotent radicals of opposite standard parabolic subgroups.
\end{lemma}
\begin{proof}
	Recall that the product morphism
	\( \Omega = U^{-} \times L \times U^{+} \to G \) is an open embedding and its image contains the unit section. We prove that actually
	\( X(A) = Y(A) \) for every open embedding
	\( Y \to X \) of pointed affine
	\( K \)-schemes of finite presentation. Let
	\( X = \Spec(R) \),
	\( \mathfrak{a} = \Ker(R \to K) \) be its augmentation ideal corresponding to the distinguished point, and
	\( \mathfrak{b} \leqt K \) an ideal corresponding to the closed subset
	\( X \setminus Y \subseteq X \). Clearly,
	\( \mathfrak{a} \) and
	\( \mathfrak{b} \) are coprime, so there exists
	\( s \in ( 1 + \mathfrak{a} ) \cap \mathfrak{b} \). The corresponding principal open subset is contained in
	\( Y \) and contains the distinguished section
	\( \Spec(K) \subseteq X \).
	
	Now fix a closed embedding
	\( X \subseteq \Affine_K^n \) such that the distinguished
	\( K \)-point is the zero section. The element
	\( s \) becomes a polynomial with unit free coefficient. If
	\( \vec{a} \colon X(A) \subseteq A^n \), then
	\(
		s( \vec{a} )
		\colon
		1 + A
		\subseteq
		( A \rtimes K )^*
	\), so
	\( \vec{a} \colon Y(A) \).
\end{proof}

\begin{lemma}
	\label{spl-uni}
	Let
	\( K \) be a semilocal commutative unital ring and
	\( G \) a unitary group scheme constructed either by a semi-regular traditional quadratic form or by a classical (unital) odd form algebra as in lemma
	\ref{twi-cla}. Let also
	\( A \) be a radical power idempotent commutative
	\( K \)-algebra in an infinitary positive category and
	\( A = \sum_{i = 1}^n A_i \) is its decomposition into a sum of power idempotent ideals. Then
	\( [ G(A), G(A) ] \leq \prod_{i = 1}^n G(A_i) \).
\end{lemma}
\begin{proof}
	Let
	\( S \) be the form algebra and
	\( (M, B, q) \) the quadratic module over
	\( S \) from the construction of
	\( G \), i.e.\ either
	\( S = K \) with
	\( \lambda = 1 \),
	\( M \) is finitely generated projective module of constant odd rank, and
	\( q \) is a semi-regular traditional quadratic form, or
	\( S = M \) is an odd form ring with standard hermitian form and
	\( q \) corresponds to its odd form parameter. Choose large
	\( \ell > 0 \) such that
	\( M \) is a direct summand of
	\( \Hyp(S^\ell) \) as in the proof of lemma
	\ref{uni-ab-k1}. Further choose
	\( h \in \ElemUnit(2 \ell, M) \) such that
	\( h M \leq \Hyp(S^\ell) \). Now if
	\( g_1, g_2 \colon G(A) \), then
	\[
		[ g_1, g_2 ]
		=
		[ g_1, h^{- 1}\, \up{h}{g_2}\, h ]
		=
		[ g_1, h^{- 1} ]\,
		\up{ g_2 h^{- 1} }{ [ g_1, h ] }
		\colon
		\Unit( 2 \ell, M \otimes_K A ).
	\]
	Actually, this expression lies both in the small unitary group
	\( \Unit(M \otimes_K A) \) and in the unrelativized elementary subgroup
	\( \ElemUnit( 2 \ell, M \otimes_K A ) \) by theorems
	\ref{ele-sub-nor} and
	\ref{unrel}. Let
	\( H(A) \) be the intersection of
	\( \ElemUnit( 2 \ell, M \otimes_K A ) \) and the group object of diagonal matrices from
	\( \Unit( 2 \ell, M \otimes_K A ) \) (the analogue of
	\( L(A) \), but for not necessarily reductive group scheme
	\( G \)).
	
	It suffices to prove that if
	\( A = A_1 + A_2 \) is a sum of two power idempotent ideals, then
	\( H(A) = H(A_1)\, H(A_2) \). We use the obvious decomposition
	\[
		\ElemUnit( 2 \ell, M \otimes_K A )
		=
		\ElemUnit( 2 \ell, M \otimes_K A_1 )\,
		\ElemUnit( 2 \ell, M \otimes_K A_2 ),
	\]
	a proof can be found in e.g.\ %
	\cite[lemma 17]{ele-loc}. Let
	\( g \colon H(A) \) and choose
	\( g_1 \colon \ElemUnit( 2 \ell, M \otimes_K A_1 ) \),
	\( g_2 \colon \ElemUnit( 2 \ell, M \otimes_K A_2 ) \) such that
	\( g = g_1 g_2 \). By lemma
	\ref{big-cell}
	\begin{align*}
			g_1
			&=
			g_{1, {-}}\, g_{ 1, \mathrm{d} }\, g_{1, {+}},
		&
			g_2
			&=
			g_{2, {+}}\, g_{ 2, \mathrm{d} }\, g_{2, {-}},
	\end{align*}
	where
	\( g_{i, {\pm}} \) are strictly upper and lower triangular matrices, and
	\( g_{ i, \mathrm{d} } \in H(A_i) \). Comparing matrix entries of the identity
	\( g_1 = g g_2^{- 1} \) we see that
	\( g_{1, {+}} g_{2, {+}} = 1 \),
	\( g_{1, {-}}\, \up{g}{ g_{2, {-}} } = 1 \), and
	\( g = g_{ 1, \mathrm{d} } g_{ 2, \mathrm{d} } \), so
	\( g \colon H(A_1)\, H(A_2) \).
\end{proof}

\begin{theorem}
	\label{sol-sem-loc}
	Let
	\( G \) be a simple group scheme over a semilocal ring
	\( K \) of local isotropic rank at least
	\( 2 \). Suppose that the absolute root system of
	\( G \) is of classical type and
	\( G \) avoids triality if the absolute root system has the type
	\( \RootSys{D}{4} \). Then
	\( \KFunc_1^G(K) \) is solvable.
\end{theorem}
\begin{proof}
	We can assume that
	\( K \) is Noetherian. Indeed,
	\( K = \bigcup_{i \in I} K_i \) is a direct union of finitely generated rings. Let
	\( S_i^{- 1} K_i \) be the
	\textit{semilocalization} of
	\( K_i \) at the image of
	\( \Max(K) \) in
	\( \Spec(K_i) \), i.e.\ %
	\[
		S_i
		=
		K_i
		\setminus
		\bigcup_{ \mathfrak{m} \in \Max(K) }
			\mathfrak{m}.
	\]
	The rings
	\( S_i^{- 1} K_i \) are also semilocal with
	\( | { \Max(K_i) } | \leq | { \Max(K) } | \) and
	\( K = \bigcup_{i \in I} S_i^{- 1} K_i \) is still a direct union representation. Thus we can replace
	\( K \) by
	\( S_i^{- 1} K_i \) with sufficiently large
	\( i \).
	
	Take an element
	\( g \in G(K)^{ (N) } \) from sufficiently high term of the derived series of
	\( G(K) \). By theorem
	\ref{sol-sem-iso} the image of
	\( g \) in
	\( G( \widehat{K}_{ \Jac(K) } ) \) lies in the elementary subgroup, where
	\(
		\widehat{K}_{ \mathfrak{a} }
		=
		\varprojlim_n K / \mathfrak{a}^n
	\) is the
	\( \mathfrak{a} \)-adic completion for
	\( \mathfrak{a} \leqt K \) because
	\[
		\widehat{K}_{ \Jac(K) }
		\cong
		\prod_{ \mathfrak{m} \in \Max(K) }
			\widehat{K}_{ \mathfrak{m} }.
	\]
	
	As in the proof of lemma
	\ref{loc-com} let
	\[
		\textstyle
		\mathfrak{a}^\infty K
		=
		\bigcap_n^{ \Pro(\Set) }
			\mathfrak{a}^n K
	\]
	for any
	\( \mathfrak{a} \leqt K \), this is an ideal of
	\( K \) in the category of pro-sets. Since
	\( \mathfrak{a} \) is finitely generated,
	\( \mathfrak{a}^\infty K \) is power idempotent. The homomorphism
	\( K \to K / \Jac(K)^\infty K \) factors through
	\( \widehat{K}_{ \Jac(K) } \), so the image of
	\( g \) in
	\( K / \Jac(K)^\infty K \) lies in the elementary subgroup. Thus we only have to prove that
	\[
		\bigl[
			G\bigl( \Jac(K)^\infty K \bigr)^{ (N) },
			h
		\bigr]
		\leq
		\Elem_G(K)
	\]
	for every
	\( h \in G(K)^{ (N) } \), where
	\( N \) is sufficiently large.
	
	Let
	\(
		\Max(K)
		=
		\{ \mathfrak{m}_1, \ldots, \mathfrak{m}_n \}
	\). For every
	\( i \) the image of
	\( h \) in
	\( G( K_{ \mathfrak{m}_i } ) \) lies in the elementary subgroup again by theorem
	\ref{sol-sem-iso}, so there exists
	\( s_i \in K \setminus \mathfrak{m}_i \) such that the image of
	\( h \) in
	\( G( K_{s_i} ) \) lies in
	\( \Elem_G( K_{s_i} ) \). Clearly, the elements
	\( s_i \) generate the unit ideal, so
	\[
		\Jac(K)^\infty K
		=
		\sum_{i = 1}^n
			( s_i \Jac(K) )^\infty K.
	\]
	All ring objects in this identity are radical and power idempotent. By lemma
	\ref{spl-uni} applied to the category
	\( \Ind( \Pro(\Set) ) \) we get
	\[
		G\bigl( \Jac(K)^\infty K \bigr)^{ (1) }
		\leq
		\prod_{i = 1}^n
			G\bigl( ( s_i \Jac(K) )^\infty K \bigr).
	\]
	
	Finally,
	\[
		\bigl[
			G\bigl( ( s_i \Jac(K) )^\infty K \bigr),
			h
		\bigr]
		\subseteq
		\bigl[ G( s_i^\infty K ), \Elem_G( K_{s_i} ) \bigr]
		\subseteq
		\Elem_G( K_{s_i}, s_i^\infty K )
		=
		\Elem_G( K, s_i^\infty K )
		\subseteq
		\Elem_G(K)
	\]
	by theorems
	\ref{ele-sub-nor} and
	\ref{unrel}.
\end{proof}

\section{Main result and concluding remarks}

We say that a Tits index
\textit{appears} in an isotropic pinning if one of components of the map
\( u \) is induced by this index.

\begin{theorem}
	\label{sol-main}
	Let
	\( K \) be a commutative unital ring and
	\( G \) a reductive group scheme over
	\( K \). Suppose that the Bass--Serre dimension of
	\( K \) is finite, the local isotropic rank of
	\( G \) is at least
	\( 2 \), and
	\( K \) contains a field if the Tits index
	\( \TitsIndex{E}{}{8}{2}{78} \) appears in a maximal isotropic pinning of
	\( G_{ \mathfrak{m} } \). Then
	\( \KFunc_1^G(K) \) is hypoabelian, i.e.\ %
	\( [ \Elem_G(K), \Elem_G(K) ] \leq G(K) \) is the largest perfect subgroup. Moreover,
	\( \KFunc_1^G(K) \) is solvable if in addition
	\begin{itemize}
		
		\item
		either
		\( G \) has an isotropic pinning of rank at least
		\( 2 \) (and
		\( K \) contains a field if
		\( \TitsIndex{E}{}{8}{2}{78} \) appears in this isotropic pinning)
		
		\item
		or
		\( G \) is simple with classical absolute root system (and avoids triality if this root system is of type
		\( \RootSys{D}{4} \)).
	
	\end{itemize}
\end{theorem}
\begin{proof}
	This follows from lemma
	\ref{sol-iso}, and theorems
	\ref{sol-dim},
	\ref{sol-sem-iso},
	\ref{sol-card},
	\ref{sol-sem-loc}. The claim about
	\( [ \Elem_G(K), \Elem_G(K) ] \) also uses theorem
	\ref{perfect}.
\end{proof}

The statement of theorem
\ref{sol-main} can be tweaked in two ways. First of all, there is an even weaker notion of ring dimension (and dimension of a not necessarily Noetherian topological space). Let us say that
\( \delta'(X) = - \infty \) if
\( X = \varnothing \) and
\( \delta'(X) \leq d \in \mathbb{N}_0 \) if there exists finite subset
\( M \subseteq X \) such that
\( \delta'(X) < d \) for all closed subsets
\( F \subseteq X \setminus M \). As for the Bass--Serre dimension, let
\( \delta'(K) = \delta'( \Max(K) ) \). Clearly,
\( \delta'(K) \leq \delta(K) \) and this inequality can be strict, e.g.\ for a ring (existing by
\cite[proposition 11]{spectra}) with
\( \Max(K) = \Real^2 \), where closed subsets are unions of finitely many lines through the origin and a finite point, as well as the whole
\( \Real^2 \). In this example
\( \delta'(K) = 1 \) and
\( \delta(K) = \jdim(K) = 2 \).

It is easy to see that
\( \delta'(K) \) satisfies all properties from theorem
\ref{bas-ser-dim}, moreover,
\( \delta'(K) \leq \delta'(E) \) if
\( K \subseteq E \) is an integral extension (so
\( \delta'(K) = \delta'(E) \) if
\( K \subseteq E \) is a finite extension). Lemma
\ref{com-dim} holds for
\( \delta' \) instead of
\( \delta \) with the same proof, and in the proof of theorem
\ref{sol-dim} we replace arbitrary
\( \mathfrak{m}_i \) from irreducible subsets of bounded combinatorial dimension to a finite subset of
\( \Max(K) \) from the inductive definition of
\( \delta'(K) \).

One can also note that our (and Bak's) definition of the dimension filtration
\( ( G^\delta(K) )_\delta \) is
\textit{impredicative}, i.e.\ it involves quantification on the proper class of all
\( K \)-algebras. We can replace
\( G^\delta(K) \) by the set
\( G^{\delta, \kappa}(K) \) of all
\( g \in G(K) \) such that the image of
\( g \) in
\( G(E) \) lies in
\( \Elem_G(E) \) for all
\( K \)-algebras
\( E \) with
\( |E| < \kappa \), where
\( \kappa \) is some infinite cardinal. All our proofs work if
\( \kappa > |K| \) and
\( \kappa > 2^{\aleph_0} = \mathfrak{c} \), because under the latter assumption the class of rings of cardinality less than
\( \kappa \) is closed under localization and formal completion. Of course,
\( G^{ \delta, \kappa }( {-} ) \) is a subfunctor of
\( G \) only on the category of
\( K \)-algebras of cardinality less than
\( \kappa \). A similar replacement can be done for cardinality filtration.

Finally, let us list several open problems.
\begin{enumerate}
	
	\item
	Show that
	\( \Elem_G(K) \) is finitely generated if
	\( K \) is finitely generated.
	
	\item
	Prove solvability of globally isotropic
	\( \KFunc_1 \) over semilocal rings for the remaining Tits index
	\( \TitsIndex{E}{}{8}{2}{78} \).
	
	\item
	Prove a uniform bound on solvability length of locally isotropic
	\( \KFunc_1 \) over semilocal rings.
	
	\item
	Find ``the smallest'' commutator word
	\( w \) of one variable such that
	\( w( \KFunc_1^G(K) ) = 1 \) for locally isotopic simple group scheme
	\( G \) over a semilocal ring.
	
	\item
	Investigate solvability of
	\( \KFunc_1^G(K) \) for isotropic
	\( G \) of rank
	\( 1 \) over a local ring.
	
	\item
	Give a construction of
	\( \Elem_G \) (or an analogue of
	\( [ \Elem_G, \Elem_G ] \)) for simple group schemes of local isotropic rank
	\( 1 \) over semilocal rings.
	
	\item
	Prove the
	\( \Elem \)-normal structure theorem for isotropic twisted forms of symplectic groups of isotropic rank at least
	\( 2 \).
	
	\item
	Classify
	\( [ \Elem, \Elem ] \)-normal subgroups of the Chevalley groups
	\( \ChevGrp^\simpcon( \RootSys{B}{2}, K ) \) and
	\( \ChevGrp^\simpcon( \RootSys{G}{2}, K ) \).
	
	\item
	Prove the
	\( \Elem \)-normal (or
	\( [ \Elem, \Elem ] \)-normal) structure theorem for locally isotropic reductive groups.
	
	\item
	Prove that relative locally isotropic
	\( \KFunc_1 \) is solvable under natural assumptions.
	
\end{enumerate}

\bibliographystyle{plain}
\bibliography{references}

\begin{thebibliography}{10}

\bibitem{nor-str-che}
E.~Abe.
\newblock Normal subgroups of {C}hevalley groups over commutative rings.
\newblock {\em Contemp. Math.}, 83:1--17, 1989.

\bibitem{tit-wei-acp}
S.~Alsaody, V.~Chernousov, and A.~Pianzola.
\newblock On the {T}its--{W}eiss conjecture and the {K}neser--{T}its conjecture
  for $ \mathsf{E}_{7, 1}^{78} $ and $ \mathsf{E}_{8, 2}^{78} $ (with an
  appendix by {R}.~{M}. {W}eiss).
\newblock {\em Forum Math. Sigma}, 9(e75):1--25, 2021.

\bibitem{k1-nil-lin}
A.~Bak.
\newblock Nonabelian $ \mathrm{K} $-theory: the nilpotent class of $
  \mathrm{K}_1 $ and general stability.
\newblock {\em $ \mathrm{K} $-Theory}, 4:363--397, 1991.

\bibitem{k1-nil-rel}
A.~Bak, R.~Hazrat, and N.~Vavilov.
\newblock Localization-completion strikes again: relative $ \mathrm{K}_1 $ is
  nilpotent by abelian.
\newblock {\em J. Pure Appl. Algebra}, 213:1075--1085, 2009.

\bibitem{stab-lin}
H.~Bass.
\newblock $ \mathrm{K} $-theory and stable algebra.
\newblock {\em Publ. Math. IH{\'E}S}, 22:5--60, 1964.

\bibitem{k-theory}
H.~Bass.
\newblock {\em Algebraic $ \mathrm{K} $-theory}.
\newblock W.~A.~Benjamin, 1968.

\bibitem{nor-str-sp}
D.~L. Costa and G.~E. Keller.
\newblock Radix redux: normal subgroups of symplectic groups.
\newblock {\em J. Reine Angew. Math.}, 427:51--105, 1992.

\bibitem{nor-str-g2}
D.~L. Costa and G.~E. Keller.
\newblock On the normal subgroups of $ \mathrm{G}_2({A}) $.
\newblock {\em Trans. Am. Math. Soc.}, 351(12):5051--5088, 1999.

\bibitem{nor-str-odd}
L.~Danilevich.
\newblock Normal structure of isotropic odd orthogonal groups.
\newblock Preprint, arXiv:2601.00763, 2026.

\bibitem{red-grp-sch}
M.~Demazure and A.~Grothendieck.
\newblock {\em S{\'e}minaire de g{\'e}om{\'e}trie alg{\'e}brique du {B}ois
  {M}arie {III}. {S}ch{\'e}mas en groupes {I}, {II}, {III}}.
\newblock Springer-Verlag, 1970.

\bibitem{uni-sem}
I.~Dias.
\newblock Unitary groups over strongly semilocal rings.
\newblock {\em Commun. Algebra}, 23(5):1797--1814, 1995.

\bibitem{kne-tit}
Ph. Gille.
\newblock Le probl{\`e}me de {K}neser--{T}its.
\newblock {\em Ast{\'e}risque}, 326:39--81, 2009.

\bibitem{k1-red-tri}
Ph. Gille and A.~Stavrova.
\newblock $ {R} $-equivalence on group schemes.
\newblock Preprint, arXiv:2107.01950v4, 2021.

\bibitem{k1-nil-uni}
R.~Hazrat.
\newblock Dimension theory and nonstable $ \mathrm{K}_1 $ of quadratic modules.
\newblock {\em $ \mathrm{K} $-Theory}, 27:293--328, 2002.

\bibitem{k1-nil-che}
R.~Hazrat and N.~Vavilov.
\newblock $ \mathrm{K}_1 $ of {C}hevalley groups are nilpotent.
\newblock {\em J. Pure Appl. Algebra}, 179:99--116, 2003.

\bibitem{spectra}
M.~Hochster.
\newblock Prime ideal structure in commutative rings.
\newblock {\em Trans. Am. Math. Soc.}, 142:43--60, 1969.

\bibitem{ort-sem}
M.~Knebusch.
\newblock Isometrien {\"u}ber semilokalen {R}ingen.
\newblock {\em Math. Z.}, 108:255--268, 1969.

\bibitem{ele-per-iso}
A.~Luzgarev and A.~Stavrova.
\newblock Elementary subgroup of an isotropic reductive group is perfect.
\newblock {\em St. Petersburg Math. J.}, 23(5):881--890, 2012.

\bibitem{stab-odd}
B.~A. Magurn, W.~van~der Kallen, and L.~N. Vaserstein.
\newblock Absolute stable rank and {W}itt cancellation for noncommutative
  rings.
\newblock {\em Invent. Math.}, 91:525--542, 1988.

\bibitem{noeth}
J.~Ohm and R.~L. Pendleton.
\newblock Rings with noetherian spectrum.
\newblock {\em Duke Math. J.}, 35(3):631--639, 1968.

\bibitem{ove-uni}
V.~Petrov.
\newblock Overgroups of unitary groups.
\newblock {\em K-Theory}, 29:147--174, 2003.

\bibitem{odd-uni-gro}
V.~Petrov.
\newblock Odd unitary groups.
\newblock {\em J. Math. Sci.}, 130(3):4752--4766, 2005.

\bibitem{ele-iso}
V.~Petrov and A.~Stavrova.
\newblock Elementary subgroups of isotropic reductive groups.
\newblock {\em St. Petersburg Math. J.}, 20(4):625--644, 2009.

\bibitem{tits-ind}
V.~Petrov and A.~Stavrova.
\newblock The {T}its indices over semilocal rings.
\newblock {\em Transform. Groups}, 16:193--217, 2011.

\bibitem{atlas}
E.~Plotkin, A.~Semenov, and N.~Vavilov.
\newblock Visual basic representations: an atlas.
\newblock {\em Int. J. Algebra Comput.}, 8(1):61--95, 1998.

\bibitem{nor-str-red}
A.~Stavrova and A.~Stepanov.
\newblock Normal structure of isotorpic reductive groups over rings.
\newblock {\em J. Algebra}, 656:486--515, 2024.

\bibitem{gen-num}
R.~G. Swan.
\newblock The number of generators of a module.
\newblock {\em Math. Z.}, 102:318--322, 1967.

\bibitem{tit-wei-tha}
M.~Thakur.
\newblock Albert algebras and the {T}its--{W}eiss conjecture.
\newblock {\em Trans. Am. Math. Soc.}, 375:6075--6091, 2022.

\bibitem{twi-for-cla}
E.~Voronetsky.
\newblock Twisted forms of classical groups.
\newblock {\em St. Petersburg Math. J.}, 34:179--204, 2023.

\bibitem{ele-loc}
E.~Voronetsky.
\newblock Locally isotropic elementary groups.
\newblock {\em Algebra i Analiz}, 36(2):1--26, 2024.
\newblock In Russian, English version is arXiv:2310.01592.

\bibitem{st-loc}
E.~Voronetsky.
\newblock Locally isotropic {S}teinberg groups {I}. {C}entrality of the $
  \mathrm{K}_2 $-functor.
\newblock {\em Algebra i Analiz}, 37(3):22--74, 2025.
\newblock In Russian, English version is arXiv:2410.14039.

\bibitem{st-schur}
E.~Voronetsky.
\newblock Locally isotropic {S}teinberg groups {II}. {S}chur multipliers.
\newblock {\em Algebra i Analiz}, 38(1):66--122, 2026.
\newblock In Russian, English version is arXiv:2507.04519.

\bibitem{weyl-ele}
E.~Voronetsky.
\newblock Weyl elements in isotropic reductive groups.
\newblock Preprint, arXiv:2601.14419, 2026.

\bibitem{bc-inj-sta}
Weibo Yu.
\newblock Stability for odd unitary $ \mathrm{K}_1 $ under the $ {\Lambda}
  $-stable range condition.
\newblock {\em J. Pure Appl. Algebra}, 217:886--891, 2013.

\bibitem{k1-nil-odd}
Weibo Yu and Guoping Tang.
\newblock Nilpotency of odd unitary $ \mathrm{K}_1 $-functor.
\newblock {\em Commun. Algebra}, 44(8):3422--3453, 2016.

\end{thebibliography}

\end{document}